\documentclass[12]{amsart}
\usepackage[utf8]{inputenc}
\usepackage{amsmath,amssymb,amsthm}
\usepackage{mathtools}
\usepackage{enumitem}
\usepackage{graphicx}
\usepackage{color}
\usepackage{amsaddr}
\usepackage{hyperref}
\usepackage[margin = 1in]{geometry}
\usepackage[all,cmtip]{xy}

\usepackage{caption}
\usepackage{subcaption}

\numberwithin{equation}{section}

\newtheorem{theorem}{Theorem}[section]
\newtheorem{proposition}[theorem]{Proposition}
\newtheorem{corollary}[theorem]{Corollary}
\newtheorem{lemma}[theorem]{Lemma}
\theoremstyle{definition}

\newtheorem{definition}[theorem]{Definition}
\newtheorem{example}[theorem]{Example}
\newtheorem{remark}[theorem]{Remark}

\newcommand{\Res}{\mathrm{Res}}
\newcommand{\vol}{\operatorname{vol}}
\newcommand{\CT}{\operatorname{CT}}
\newcommand{\indeg}{\mathrm{indeg}}
\newcommand{\aff}{\mathrm{aff}}
\newcommand{\outdeg}{\mathrm{outdeg}}

\title[Refinements and Symmetries of the Morris identity for volumes of flow polytopes]{Refinements and Symmetries of the Morris \\identity for volumes of flow polytopes}

\author[Alejandro H. Morales, William Shi]{Alejandro H. Morales$^\star$, \ and \ \ William Shi$^\dagger$}

\thanks{\today}
\thanks{$^\star$Department of Mathematics and Statistics, UMass, Amherst, MA~01003.\quad
Email: \href{mailto:ahmorales@math.umass.edu}{ahmorales@math.umass.edu}}
\thanks{$^\dagger$Northview High School, Johns Creek, GA~30097.\quad Email: \href{mailto:williamshi080@gmail.com}{williamshi080@gmail.com}}

\begin{document}

\maketitle 

\begin{center}
\dedicatory{\em Dedicated to Doron Zeilberger in recognition of his 70th birthday.}
\end{center}

\begin{abstract}
  Flow polytopes are an important class of polytopes in combinatorics whose lattice points and volumes have interesting properties and relations. The Chan--Robbins--Yuen (CRY) polytope is a flow polytope with normalized volume equal to the product of consecutive Catalan numbers. Zeilberger proved this by evaluating the Morris constant term identity, but no combinatorial proof is known. There is a refinement of this formula that splits the largest Catalan number into Narayana numbers, which M\'esz\'aros gave an interpretation as the volume of a collection of flow polytopes. We introduce a new refinement of the Morris identity with combinatorial interpretations both in terms of lattice points and volumes of flow polytopes. Our results generalize M\'esz\'aros's construction and a recent flow polytope interpretation of the Morris identity by Corteel--Kim--M\'esz\'aros. We prove the product formula of our refinement following the strategy of the Baldoni--Vergne proof of the Morris identity. Lastly, we study a symmetry of the Morris identity bijectively using the Danilov--Karzanov--Koshevoy triangulation of flow polytopes and a bijection of M\'esz\'aros--Morales--Striker.
\end{abstract}

\section{Introduction} \label{sec: intro}

\subsection{Foreword}
Flow polytopes play a fundamental role in combinatorial optimization through their relation to maximum matching and minimum cost problems (e.g. see \cite[Ch. 13]{schrijver2003combinatorial}). Flow polytopes have been used in various fields like toric geometry \cite{hille} and representation theory \cite{BVkpf}. More recently, they have been related to geometric and algebraic combinatorics thanks to connections with   Schubert polynomials \cite{escobar2018subword}, diagonal harmonics \cite{tesler}, Gelfand-Tsetlin polytopes \cite{flowsGT}, and generalized permutahedra \cite{permutahedra}.

Given a graph $G$ with vertex set $\{0, 1, \ldots, n, n+1\}$ and edges $(i,j)$ oriented $i \to j$ if $i < j,$ we associate with $G$ a net flow vector $\mathbf{a} = (a_0,a_1,\dots, a_{n}, -\sum_{i=0}^{n} a_i)$ such that vertex $i$ has net flow $a_i$ for $i = 0, 1, \ldots, n.$ The set of all flows with net flow vector $\mathbf{a}$, called the \textit{flow polytope}, is denoted by $\mathcal F_G(\mathbf{a}).$ Define $K_G(\mathbf{a})$ as the number of lattice points (integer flows) of $\mathcal F_G(\mathbf{a}),$ called the \emph{Kostant vector partition function}. The name comes from the fact that for the complete graph $k_{n+2}$, $K_{k_{n+2}}({\bf a})$ is a vector partition function studied by Kostant in the context of Lie algebras (e.g. \cite{Hum}). The following theorem, which appears in unpublished work of Postnikov and Stanley and in the work of Baldoni-Vergne \cite{BVkpf}, relates the volume of a flow polytope to a Kostant partition function. See Section~\ref{sec:vol proof} for a new  recursive proof of this result.  

\begin{theorem}[Postnikov-Stanley, Baldoni-Vergne \cite{BVkpf}] \label{thm:vol-kpf}
For a loopless digraph $G$ with vertices $\{0, 1, \ldots, n+1\}$ having unique source $0$ and unique sink $n + 1$,
\begin{equation} \label{eq:vol is kpf}
\vol \mathcal F_G(1,0,\ldots,0,-1) = K_G (0,d_1, \dots, d_{n}, -\sum_{i=1}^{n} d_i),
\end{equation}
where $d_i = \indeg_G(i) - 1.$
\end{theorem}

An important example of a flow polytope is the Chan-Robbins-Yuen (CRY) polytope \cite{CRY}, defined as $CRY_{n+1} := \mathcal F_{k_{n+2}}(1,0,\dots,0,-1).$ Zeilberger calculated the volume of $CRY_{n+1}$ algebraically using the Morris constant term identity, equivalent to the famous Selberg integral formula (see \cite{Forrester}). For convenience, we use the term \textit{volume} in this paper to refer to normalized volume.

\begin{theorem}[Zeilberger's Morris Identity \cite{Z}] \label{morris}
For positive integers $n,a,$ and $b,$ and nonnegative integer $c,$ define the constant term
\begin{equation*}
    M_n(a,b,c) := \CT_x \prod_{i=1}^n (1-x_i)^{-b}x_i^{-a+1}\prod_{1 \leq i < j \leq n}(x_j - x_i)^{-c},
\end{equation*}
where $\CT_x := \CT_{x_n}\cdots \CT_{x_1}$. Then 
\begin{equation} \label{eq:morris}
    M_n(a,b,c) = \prod_{j=0}^{n-1} \frac{\Gamma(a-1 + b + (n-1+j)\frac{c}{2})\Gamma(\frac{c}{2}+1)}{\Gamma(a+j\frac{c}{2})\Gamma(b + j\frac{c}{2})\Gamma(\frac{c}{2}(j+1)+1)}.
\end{equation}
\end{theorem}

By specializing this identity, Zeilberger proved that the volume of $CRY_{n+1}$ is the product of the first $n - 1$ Catalan numbers. 

\begin{theorem}[Zeilberger \cite{Z}] \label{thm:CRY}
The volume of the polytope $CRY_{n+1}$ is given by $M_{n}(1,1,1) = \prod_{i=1}^{n-1} C_i,$ where $C_i = \frac{1}{i+1}\binom{2i}{i}$ is the $i$th Catalan number.
\end{theorem}

Despite the numerous interpretations of $C_n$, no combinatorial proof of Theorem~\ref{thm:CRY} is known. Corteel-Kim-M\'esz\'aros \cite[Theorem 1.2]{CKM} also showed that for any positive $a,b,$ and $c,$ $M_n(a,b,c)$ gives the volume of the flow polytope on the following graph. For positive integer $n$, let $k_{n+2}^{a,b,c}$ denote the graph on vertex set $\{0, 1, \ldots, n+1\}$ with edge $(0,i),$ $i \in [1,n]$ appearing with multiplicity $a,$ edge $(i, n+1)$, $i \in [1,n],$ appearing with multiplicity $b,$ and $(i,j), 1 \leq i < j \leq n,$ appearing with multiplicity $c$ (see Figure~\ref{fig:knabc}). Note that $k_{n+2}=k_{n+2}^{1,1,1}$. Then they showed the following.

\begin{theorem}[{Corteel-Kim-M\'esz\'aros \cite{CKM}}] \label{thm:ckm}
Let $n,a$ and $b$ be positive integers, $c$ be a nonnegative integer, and let $a_i = a-1 + c(i-1)$ for $i=1,\ldots,n$. Then
\begin{equation} 
\label{eq: Mn is kpf and volume} \vol \mathcal F_{k_{n+2}^{a,b,c}}(1,0, \ldots, 0, -1) = K_{k_{n+2}^{a,b,c}}(0,a_1,\ldots,a_n,-\sum_{i=1}^n a_i) =  M_n(a,b,c). 
\end{equation}
\end{theorem}

From the product formula in \eqref{eq:morris} it follows that  $M_n(a,b,c)$ is symmetric in $a$ and $b$. This is less clear from the volume and lattice point interpretation of $M_n(a,b,c)$ in \eqref{eq: Mn is kpf and volume}. 

In addition, there is an interesting refinement of the volume formula $M_n(1,1,1)$ of the CRY polytope. Namely, the following conjecture of Chan-Robbins-Yuen \cite[Conj. 2]{CRY} settled by Zeilberger \cite{Z}, refines the product $C_{n}C_{n-1}\cdots C_1$ by splitting $C_{n}$ into a sum of Narayana numbers $N(n,k) = \frac{1}{n}\binom{n}{k}\binom{n}{k-1}$. The original conjecture used the Kostant partition function interpretation and M\'esz\'aros \cite[Thm. 11]{Mproduct} then gave a geometric interpretation of this refinement by providing a collection of interior disjoint polytopes whose volumes equal $N(n,k) \prod_{i=1}^{n-1}C_i$. To state these interpretations, we introduce some notation. Given a graph $k_{n+2}^{a,b,c}$ as above and a $k$-element set $S\in \binom{[n]}{k}$, let $k_{n+2}^{a,b,c}(S)$ be the graph obtained from taking $k_{n+2}^{a,b,c}$, adding $n$ edges $(0,n+1)$, and for each $i\in S$ deleting one of the $a$ incoming edges $(0,i)$ and adding an outgoing edge $(i,n+1)$ (See Figure~\ref{fig:knabc}).

\begin{theorem}[Zeilberger \cite{Z}; M\'esz\'aros \cite{Mproduct}] \label{conj 2 cry}
For a positive integer $n$ and a nonnegative integer $k\leq n$, the product $N(n,k) \prod_{i=1}^{n-1}C_i$ equals the following:
\begin{itemize}
    \item[(i)] The sum of Kostant partition functions $K_{k_{n+2}}(0, a_1, \dots, a_{n}, -\sum_{j=1}^{n} a_j)$ such that for $i \in [n],$ $a_i \leq i-1,$ with $a_i = i-1$ holding for exactly $k$ values of $i$.
    \item[(ii)] The volume of the interior disjoint polytopes $\{\mathcal{F}_{k_{n+2}(S)} \mid S \in \binom{[n]}{k}\}$.
\end{itemize}
\end{theorem}

In \cite{Z}, Zeilberger sketched the proof of Theorem~\ref{conj 2 cry} using Aomoto's refinement of the Selberg integral \cite{aomoto}, but no explicit refinement of $M_n(a,b,c)$ was given (see also \cite{Zletter}).

The aims of this paper are threefold: give such a refinement of  $M_n(a,b,c)$, with a product formula that implies Theorems~\ref{morris} and \ref{thm:CRY}, provide a geometric and lattice point interpretations of the refinement extending Theorems~\ref{conj 2 cry} and \ref{thm:ckm}, and lastly study the symmetry and new relations of $M_n(a,b,c)$. We next describe our main results.

\begin{figure}
    \centering
    \includegraphics[scale=0.8]{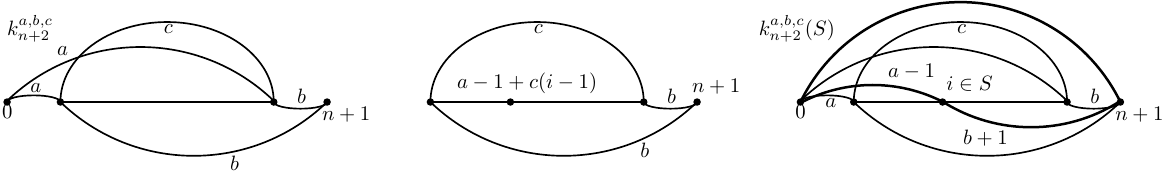}
    \caption{The graph $k_{n+2}^{a,b,c}$, the graph of the Kostant partition function corresponding to the volume of $\mathcal F_{k_{n+2}^{a,b,c}}$, and the graph of one of the polytopes corresponding the volume interpretation of $\Psi_{n}(k,a,b,c)$.} 
    \label{fig:knabc}
\end{figure}

\subsection{A new refinement of \texorpdfstring{$M_n(a,b,c)$}{M_n(a,b,c)}}

 Our refinement is inspired by a related refinement of $M_n(a,b,c)$ introduced by Baldoni-Vergne \cite{BV} to prove the Morris identity (Theorem~\ref{morris}), for which we extend a Kostant partition function interpretation (see Section~\ref{sec: phi}) but which did not imply Theorem~\ref{conj 2 cry} as a special case.

To state our results, define the constant term
\begin{equation} \label{eq: definition Psi}
    \Psi_n(k,a,b,c) := \CT_x[t^k]\prod_{i=1}^n (1-x_i)^{-b}x_i^{-a+1}(1+t\frac{x_i}{1-x_i}) \prod_{1 \leq i < j \leq n} (x_j - x_i)^{-c}.
\end{equation}

In the case that $k = 0,$ $\Psi_n(0,a,b,c) = M_n(a,b,c)$. We now give Kostant partition function and polytope volume interpretations for $\Psi_n(k,a,b,c),$ as well as an explicit product formula.

\begin{theorem} \label{psi kpf}
For positive integers $n, a,$ and $b,$ nonnegative integer $c,$ and nonnegative integer $k \leq n,$ the constant term $\Psi_n(k,a,b,c)$ equals the following:
\begin{enumerate}
    \item[(i)]  the sum of Kostant partition functions of the form $K_{k_{n+2}^{a,b,c}}(0, a_1, \dots, a_{n}, -\sum_{j=1}^n a_j)$ such that for $i \in [n],$ $a_i \leq a - 1 + c(i-1),$ with $a_i = a - 1 + c(i-1)$ holding for exactly $n-k$ values of $i.$
    \item[(ii)] the volume of the interior disjoint polytopes $\{\mathcal{F}_{k^{a,b,c}_{n+2}(S)} \mid S \in \binom{[n]}{k}\}$. Thus,
\[
\Psi_n(k,a,b,c) = \sum_{S \in \binom{[n]}{k}} \vol  \mathcal{F}_{k^{a,b,c}_{n+2}(S)}.
\]
\end{enumerate}
\end{theorem}

We see that when $a = b = c = 1,$ the Kostant partition function interpretation of $\Psi_n(k,a,b,c)$ reduces to Theorem~\ref{conj 2 cry}, giving that
$\Psi_n(k,1,1,1) =N(n,k) \prod_{i=1}^{n-1}C_i$. As a corollary, our constant term $\Psi_n(k,a,b,c)$ refines the Morris constant term  $M_n(a,b+1,c)$. 

\begin{corollary} \label{cor:M refinement} Let $n,a,$ and $b$ be positive integers, and let $c$ be a nonnegative integer. Then
\begin{equation} \label{eq: Psi refine M}
    M_n(a,b+1,c) = \sum_{k=0}^n \Psi_n(k,a,b,c).
\end{equation}
\end{corollary}

We also compute the following explicit product formula for $\Psi_n(k,a,b,c)$ that completes our refinement and new proof of the Morris identity.

\begin{theorem}\label{thm:psi product}
For positive integers $n, a,$ and $b,$ nonnegative integer $c,$ and nonnegative integer $k \leq n,$ the constant term $\Psi_n(k,a,b,c)$ is given by
\begin{equation} \label{eq: product formula Psi}
    \Psi_n(k,a,b,c) = \binom{n}{k}M_n(a,b,c)\prod_{j=1}^{k}\frac{a-1+(n-j)\frac{c}{2}}{b+(j-1)\frac{c}{2}}.
\end{equation}
\end{theorem}

We show Theorem \ref{thm:psi product} by proving four recurrence relations satisfied by $\Psi_n(k,a,b,c)$, by proving these relations uniquely define $\Psi_n(k,a,b,c),$ and by proving the product formula also satisfies these relations. This closely follows the approach of Baldoni-Vergne \cite[p. 10]{BV} in their proof of the Morris identity. However, our proofs are combinatorial rather than algebraic, with the notable exception of the proof of the relation \eqref{psi rel 4}, which after a reformulation states that for $1 \leq k \leq n$, 
\[\label{eq: key} 
k(b+(k-1)c/2)\cdot \Psi_{n}(k,a,b,c) = (n-k+1)(a-1+(n-k)c/2) \cdot \Psi_{n}(n-k+1,b+1,a-1,c).\tag{$\star$}\]
We leave as an open problem to prove this relation combinatorially, which would then imply a combinatorial proof of the volume formula for the CRY polytope (Theorem~\ref{thm:CRY}).

\subsection{A fundamental symmetry of \texorpdfstring{$M_n(a,b,c)$}{M_n(a,b,c)}}

We also explain the symmetry $M_n(a,b,c)=M_n(b,a,c)$ with the volume and lattice point interpretations of \eqref{eq: Mn is kpf and volume}. In particular, we use a triangulation of flow polytopes of Danikov-Karzanov-Koshevoy \cite{DKK} and a correspondence from \cite{MMS} to give a bijection between the lattice points of $\mathcal{F}_{k_{n+2}^{a,b,c}}(0,a_1,a_2,\ldots,-\sum_i a_i)$ and $\mathcal{F}_{k_{n+2}^{b,a,c}}(0,b_1,b_2,\ldots,-\sum_i b_i)$ where $a_i = a-1+c(i-1)$ and $b_i=b-1+c(i-1)$ for $i=1,\ldots,n$. The bijection holds for any graph $G$ and as a special case we obtain a bijection of Postnikov \cite{Pos} between lattice points of $(p-1)\Delta^{q-1}$ and $(q-1)\Delta^{p-1}$ further studied in \cite{GNP}.

\subsection{Outline} The rest of this paper is structured as follows. In Section~\ref{sec: bg}, we establish basic theory surrounding flow polytopes, Kostant partition functions, and the Morris constant term identity. This includes closed formulas and asymptotics for special cases of $M_n(a,b,c)$. Section~\ref{sec:vol proof} gives a new recursive proof of Theorem~\ref{thm:vol-kpf} by extending a well-known subdivision relation of flow polytopes to integer flows. In Section~\ref{sec: symmetry} we give three proofs of the symmetry $M_n(a,b,c)=M_n(b,a,c)$ including a bijection between lattice points of two flow polytopes. In Section~\ref{sec: psi},  we prove our results for $\Psi_n(k,a,b,c)$, including Theorem~\ref{psi kpf}, Corollary~\ref{cor:M refinement}, and Theorem~\ref{thm:psi product}. In Section~\ref{sec: phi}, we apply our methods for $\Psi_n(k,a,b,c)$ to the Baldoni-Vergne constant term and prove Theorem~\ref{phi kpf}, and in Section~\ref{sec: remarks} we provide final remarks and some open questions.

\section{Background and Notation} \label{sec: bg}

\subsection{Flow polytopes and their subdivisions}

Given a loopless acyclic connected digraph $G$ with vertex set $\{0,1,\ldots,n,n+1\}$ and $m$ edges, we orient edge $(i,j)$  from $i$ to $j$ if $i < j.$ We can then represent each edge $(i,j)$ by the positive type $A_n$ root $\alpha(i,j) = e_i - e_j$. We also define $M_G$ to be the $(n+2) \times m$ matrix whose columns are given by the multiset $\{\{\alpha(e)\}\}_{e \in E(G)}.$ 

Then given a net flow vector $\mathbf{a} = (a_0,a_1,\dots, a_n, -\sum_{i=0}^{n} a_i)$, where $a_i$ represents the net flow at vertex $i$, we define an \textit{$\mathbf{a}$-flow} $\mathbf{f}_G$ as a vector $\mathbf{f}_G = (f(e))_{e \in E(G)}$ satisfying $M_g\mathbf{f}_G = \mathbf{a}.$ We now define the \textit{flow polytope} $\mathcal F_G(\mathbf{a})$ as the set of all $\mathbf{a}$-flows on $G.$ More precisely, $\mathcal F_G(\mathbf{a}) := \{\mathbf{f}_G \in \mathbb{R}^m_{\geq 0}\mid M_G\mathbf{f}_g = \mathbf{a}\}.$ In the absence of an explicit vector $\mathbf{a}$, it is implied that $\mathbf{a} = (1, 0, \ldots, 0, -1).$ In other words, $\mathcal{F}_G := \mathcal{F}_G(1,0, \ldots, 0, -1).$ If $G$ has a unique source $0$ and sink $n+1$, then the \textit{dimension} of $\mathcal{F}_G$ is $m-n-1$. The vertices of $\mathcal{F}_G$ are given by unit flows along maximal directed paths from the source to the sink called \textit{routes}. 

Next we define a notion of equivalence for flow polytopes. Let $\aff(\cdot)$ denote affine span. For two flow polytopes $P \subset \mathbb R^{n}$ and $Q \subset \mathbb R^{m},$ we say that $P$ and $Q$ are \textit{integrally equivalent}, denoted $P \equiv Q,$ if there exists an affine transformation $\varphi: \mathbb R^{n} \to \mathbb R^{m}$ that is a bijection both when restricted between $P$ and $Q$ and when restricted between $\aff(P) \cap \mathbb Z^{n}$ and $\aff(Q) \cap \mathbb Z^m.$ Polytopes that are integrally equivalent share many similar properties, including the same volume and Ehrhart polynomials.

For a digraph $G$ as above, we denote by $G^r$ the digraph with the same vertices and edges $E(G^r)=\{(i,j) \mid (n+1-j,n+1-i) \in E(G)\}$. That is, the digraph obtained from $G$ by reversing the edges and relabeling the vertices $i\mapsto n+1-i$. By reversing the flows, one shows that the flow polytopes of $G$ and $G^r$ with netflow $(1,0,\ldots,0,-1)$ are integrally equivalent. 

\begin{lemma} \label{lem: reverse flow polytopes}
For a loopless digraph $G$ with vertices $\{0,1,\ldots,n+1\}$  having a unique source $0$ and unique sink $n+1$ then $\mathcal{F}_G \equiv \mathcal{F}_{G^r}$.
\end{lemma}

We now give a recursive subdivision of flow polytopes used by Postnikov-Stanley in their unpublished work. See also \cite[Section 4]{MM}. 

Let $G = (\{0, 1, \ldots, n, n+1\}, E)$. We now repeatedly apply the following algorithmic step, called the \textit{reduction rule}: starting with a graph $G_0$ on vertex set $
\{0, 1, \ldots, n, n+1\}$ and $(i,j),(j,k) \in E(G_0)$ for some $i < j < k,$ we reduce $G_0$ to two graphs $G_1$ and $G_2$ with vertex set $
\{0, 1, \ldots, n, n+1\}$ and edge sets \begin{align}
    \label{reduction1} E(G_1) & := E(G_0)\setminus\{(j,k)\}\cup \{(i,k)\}, \\
    E(G_2) & := E(G_0) \setminus\{(i,j)\} \cup \{(i,k)\} \label{reduction2}.
\end{align}

\begin{proposition}[{Subdivision Lemma, Postnikov, Stanley \cite{stanley} (e.g. \cite[Prop. 1]{Mproduct})}] \label{reduction proposition}
Given a graph $G_0$ on the vertex set $
\{0, 1, \ldots, n, n+1\}$ and $(i,j),(j,k) \in E(G_0)$ for arbitrary $i < j < k,$ define $G_1$ and $G_2$ by the above reduction rule. Then we have
\begin{equation*}
    \mathcal F_{G_0} \equiv \mathcal F_{G_1} \cup \mathcal F_{G_2}, \qquad \mathcal F^{\circ}_{G_1} \cap \mathcal F^{\circ}_{G_2} = \varnothing,
\end{equation*}
where $\mathcal{F}^{\circ}_G$ denotes the interior of the polytope $\mathcal F_G$.
\end{proposition}

The subdivision lemma is illustrated in Figure~\ref{fig:subdiv}. The proof can be found in \cite{MM}. Define a graph $G$ to be \textit{reducible} if we can apply the reduction rule to two of its edges (that is, there exists $(i,j), (j,k) \in E(G)$). Otherwise, the graph $G$ is \textit{irreducible}. We now define the reduction tree $\mathcal T(G)$ of a graph $G$. The root of $\mathcal T(G)$ is $G,$ and each node $G_0$ has two children $G_1$ and $G_2$ described by the reduction rule. Each leaf of $\mathcal T(G)$ is hence irreducible. $\mathcal T(G)$ is not unique and depends on the order of reductions applied, but the number of leaves is always the same. 

\begin{figure}
    \centering
    \includegraphics[scale=0.8]{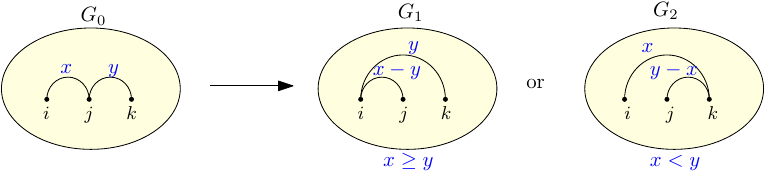}
    \caption{The subdivision lemma reduces a flow polytope to two interior disjoint polytopes whose union is integrally equivalent to the original flow polytope.}
    \label{fig:subdiv}
\end{figure}

\subsection{Kostant partition functions} \label{sec: kpf} We now examine the lattice points of $\mathcal F_G(\mathbf{a})$, i.e. the integer flows. For a graph $G$ on vertex set $
\{0, 1, \ldots, n, n+1\}$ and $(i,j)$ oriented $i \to j$ if $i < j,$ denote by $\mathcal F^\mathbb{Z}_G(\mathbf{a})$ the set of lattice points of the flow polytope $\mathcal F_G(\mathbf{a}),$ and define $K_G(\mathbf{a}):=\# F^\mathbb{Z}_G(\mathbf{a})$ to be the number of such lattice points, called the \textit{Kostant partition function}. The name comes from interpreting the function in the case of $G=k_{n+2}$ as giving the number of ways of writing $\mathbf{a}$ as a $\mathbb{N}$-linear combination of the type $A$ positive roots $e_i - e_j$, where $e_i$ is the $i$th standard vector and $i < j$. In the theory of semisimple Lie algebras there are classical formulas for weight and tensor product multiplicities in terms of $K_{k_{n+2}}({\bf a})$ (see \cite[Section 24]{Hum}).

The generating function of Kostant partition functions on $G$ is given by \begin{equation*}
    \sum_{\bf a} K_G({\bf a}) {\bf x}^{\bf a} = \prod_{(i,j)\in E(G)} (1-x_ix_j^{-1})^{-1},
\end{equation*}

where the term $x_ix_j^{-1}$ represents a single flow along the edge $(i,j)$, and the number of flows with net flow of $j$ at vertex $i$ is represented by the coefficient of $x_i^j.$ In particular, for the graph $k_{n+2}^{a,b,c},$ the generating function can be simplified to
\begin{equation} \label{eq: kpf gen function kabc}
    \sum_{\bf a} K_{k_{n+2}^{a,b,c}}({\bf a}) {\bf x}^{\bf a} = \prod_{i=1}^n (1-x_0x_i^{-1})^{-a}(1-x_ix_{n+1}^{-1})^{-b} \prod_{1 \leq i < j \leq n} (1-x_ix_j^{-1})^{-c}.
\end{equation}

Theorem~\ref{thm:vol-kpf} relates the volume of a flow polytope to a Kostant partition function with a certain net flow vector. Using the generating function for Kostant partition functions, this has very useful implications, such as Theorem~\ref{thm:ckm}. To prove Theorem~\ref{thm:ckm}, first apply Theorem~\ref{thm:vol-kpf} for $k^{a,b,c}_{n+2}$. Since the net flow at the source is zero, we can ignore the term $\prod_{i=1}^n (1-x_0x_i^{-1})^{-a}$. Because the total flow is conserved, the flow at vertex $n+1$ is already determined, so we can simplify the product by setting $x_{n+1} = 1.$ The result follows by extracting the appropriate coefficient in \eqref{eq: kpf gen function kabc}, and expressing it as a constant term extraction (see \cite[Theorem 1.2]{CKM}). This approach thus gives a way to express Kostant partition functions as a constant term.

\subsection{Catalan numbers, Narayana numbers, and Proctor's formula}

The Catalan numbers satisfy the formula $C_n = \frac{1}{n+1}\binom{2n}{n},$ and are one of the most ubiquitous sequences in combinatorics. For instance, the Catalan number $C_n$ counts more than 200 different combinatorial objects \cite{Catalanbook}. The Catalan numbers are refined by the Narayana numbers $N(n,k) = \frac{1}{n}\binom{n}{k}\binom{n}{k-1}$ such that \begin{equation*}
    C_n = \sum_{k=1}^n N(n,k).
\end{equation*}
In analogy to the Catalan numbers, the Narayana number $N(n,k)$ counts the number of lattice paths from $(0,0)$ to $(n,n)$ that do not pass above the line $y = x$ and has $2k-1$ turns. Notably, both Narayana and Catalan numbers appear in Theorem~\ref{conj 2 cry}, where the Narayana refine the volume of the CRY polytope. Proctor's formula describes another form in which Catalan numbers  appear. In \cite{proctor}, Proctor shows that
\begin{equation*}
\prod_{1 \leq i < j \leq n} \frac{2(a-1)+i+j-1}{i+j-1} = \det[C_{n-2+i+j}]_{i,j=1}^{a-1}.
\end{equation*}
We will see Catalan numbers appear in several forms in Section~\ref{sec: Morris} for special cases of the Morris identity, including through Proctor's formula.  

\subsection{The Morris constant term identity \texorpdfstring{$M_n(a,b,c)$}{M_n(a,b,c)}} \label{sec: Morris} We first formalize the notion of a constant term extraction. For a Laurent series $f(x_i)$, we denote the coefficient of $x_i^j$ by $[x_i^j]f(x_i),$ and we denote the constant term in $x_i$ by $\CT_{x_i}f(x_i)$. Similarly, for a Laurent series $f(x_1,x_2, \ldots, x_n),$ we denote the constant term by $\CT_x f(x_1,x_2, \ldots, x_n):= \CT_{x_n}\cdots\CT_{x_1}f(x_1,x_2, \ldots, x_n)$.

Similary, we define the \textit{residue} of $f(x_i)$ with respect to $x_i$ as the coefficient of $x_i^{-1}.$ We denote this by $\Res_{x_i} f(x_i):= [x_i^{-1}] f(x_i)$, and we also use the notation $\Res_x f(x_1, \ldots, x_n) := \Res_{x_n} \cdots \Res_{x_1} f(x_1, \ldots, x_n).$ A useful property of residues is that for a meromorphic function $f(x_1, x_2, \ldots x_n),$ the residue of a partial derivative is always zero. That is, 
$$\Res_{x_i} \frac{\partial}{\partial x_i} f(x_1, x_2, \ldots x_n) = 0.$$ 

We now give some special properties and cases of the Morris constant term identity \eqref{eq:morris}. Note that for $c>0,$ we can substitute $\Gamma(x+1) = x \Gamma(x)$ to obtain the following alternate form of Morris identity
\begin{equation*}
    M_n(a,b,c) = \frac{1}{n!} \prod_{j=0}^{n-1} \frac{\Gamma(a-1 + b + (n-1+j)\frac{c}{2})\Gamma(\frac{c}{2})}{\Gamma(a+j\frac{c}{2})\Gamma(b + j\frac{c}{2})\Gamma(\frac{c}{2}(j+1))}.
\end{equation*}

This form of the Morris identity is used in most of our computational proofs. 

Recall that $M_n(1,1,1)$ is a product of consecutive Catalan numbers. Interestingly, the case $M_n(a,1,1)$ strongly resembles $M_n(1,1,1),$ and is, by Proctor's formula, a product of Catalan numbers times a determinant of Catalan numbers. 

\begin{corollary} \cite{BVkpf, Mproduct}
The constant term $M_n(a,1,1)$ can be expressed as a product of consecutive Catalan numbers times a determinant of Catalan numbers.
\begin{align*}
M_n(a,1,1) & = C_1C_2 \cdots C_{n-1}\prod_{1 \leq i < j \leq n} \frac{2(a-1) + i+ j-1}{i+j-1} = C_1C_2 \cdots C_{n-1}\cdot \det[C_{n-2+i+j}]_{i,j=1}^{a-1}.
\end{align*}
\end{corollary}

Next, we list simplified identities for some other special cases of the Morris identities. Proofs of these formulas and other special cases, namely $M_n(a,b,1)$ and $M_n(a,b,2k)$, are rather computational and are hence provided in the Appendix. Intriguingly, the explicit formula for $M_n(a,b,1)$ strongly resembles the formula for $M_n(a,1,1).$

\begin{corollary} \label{cor: Mn(a,b,1)} For positive integers $n,a,$ and $b$, the constant term $M_n(a,b,1)$ is given by
\begin{align*}
M_{2n}(a,b,1) & = C_1C_2 \cdots C_{2n-1} \prod_{1 \leq i < j \leq 2n} \frac{2(a+b-2)+i+j-1}{i+j-1} \prod_{1 \leq i \leq n} \frac{\binom{2a + 2b+ 4i -2}{2a + 2i -1}}{\binom{2a + 2b + 4i - 2}{2i+1}}\\
M_{2n-1}(a,b,1) & = \binom{a+b-2}{a-1}C_1C_2\cdots C_{2n-2} \prod_{1 \leq i < j \leq 2n -1} \frac{2(a+b-2)+i+j-1}{i+j-1} \prod_{1 \leq i \leq n-1} \frac{\binom{2a + 2b + 4i-4}{2a + 2i - 2}}{\binom{2a + 2b + 4i - 4}{2i}}.
\end{align*}
\end{corollary}

By expressing the above special cases in terms of \emph{superfactorials}, we also give in the Appendix asymptotic results for the following values of the Morris identity: $M_n(1,1,1)$, $M_n(n,1,1)$ and $M_n(n,n,1)$. 

Lastly, we also give a formula for $M_n(a,b,c)$ for even $c$, which curiously differs significantly from other computed special cases.

\begin{corollary}\label{cor: Mn(a,b,2k)} For positive integers $n,a,b$ and $k,$ the constant term $M_n(a,b,2k)$ is given by the product
\begin{equation*}
    M_n(a,b,2k) = \prod_{i=1}^n \frac{(a+b-2+(2i-3)k)!k!}{((i-2)k)!(ik)!}\binom{a + b-2+(2i-2)k}{a-1 + (i-1)k}.
\end{equation*}
\end{corollary}

\section{A recursive proof of Theorem~\ref{thm:vol-kpf}} \label{sec:vol proof}

In this section we give a new recursive proof of Theorem~\ref{thm:vol-kpf} by introducing a subdivision map for the right-hand side of \eqref{eq:vol is kpf}. To give our proof, we first show that all subdivisions reduce to a similar form.

\begin{lemma} \label{base case algorithm}
Every connected directed graph $G$ on vertex set $\{0,1, \ldots, n+1\}$ with unique source $0$ and unique sink $n+1$ can be reduced to subdivisions $G'$ with the same vertex set, unique source and sink and for $i \in [n],$ $\outdeg_{G'}(i) = 1.$
\end{lemma}

\begin{proof}
We apply the following algorithm:
\begin{enumerate}
    \item Consider if graph $G$ has a non-empty set $S$ of vertices $i$ such that $\indeg_G(i) > 1$ and $\outdeg_G(i) >1.$ Then we apply the reduction rule at any vertex in $S.$ 
    \item Consider if graph $G$ has a non-empty set $T$ of vertices $i$ such that $\indeg_G(i) = 1$ and $\outdeg_G(i) >1.$ Then we apply the reduction rule at any vertex in $T.$
    
    We note the net flow for a vertex in $T$ is zero, so the flow along the incoming edge must be at least the flow along any of the outgoing edges. Obtaining $G_1$ and $G_2$ as in \eqref{reduction1} and \eqref{reduction2}, applying the map on flows in the subdivision as shown in Figure~\ref{fig:subdiv}, gives that $\mathcal F_{G_1} = \varnothing,$ and can be disregarded. Hence, we see that the uniqueness of the sink and source are also preserved in $G_2$.
    \item We continually apply steps (1) and (2) until $S = T = \varnothing$, at which point we conclude $\outdeg_G(i) = 1$ for $i \in [n].$  
\end{enumerate}
Since the graph is finite, we see the algorithm must terminate.
\end{proof}

We now prove the following lemma, which establishes the base case for our induction.

\begin{lemma}[Base Case] \label{base case} For a graph $G$ on vertex set $\{0,1, \ldots, n+1\}$ with $m$ edges, unique source $0$, unique sink $n+1$, and where $\outdeg_G(i) = 1$ for $i \in [n],$ we have that
\begin{equation*}
    \vol \mathcal F_{G}(1,0,\ldots,0,-1) = K_{G} (0,d_1, \dots, d_{n}, -\sum_{i=1}^{n} d_i) = 1,
\end{equation*}
where $d_i = \indeg_G(i) - 1.$ 
\end{lemma}

\begin{proof}
First we show $\vol \mathcal F_{G}(1,0, \ldots, 0,-1) = 1.$ Since $\outdeg_G(i)=1$ for $i \in [n]$ then the source has outdegree $m - n$, and the flows along these $m-n$ edges determines a unique flow on $G$. To see this, note that the flows of the outgoing edges of vertices in the set $\{0,1, \ldots,i\}$ for $i \in [n]$ determine recursively the outgoing flow at vertex $i+1$. We see that $\mathcal F_G$ is integrally equivalent to a $(m-n-1)$-dimensional simplex and has normalized volume 1.

Next we show that  $K_{G} (0,d_1, \dots, d_{n}, -\sum_{i=1}^{n} d_i)=1.$ We recursively show that there is only one integer flow $f$ with the desired net flow. Since the source has net flow zero, then $f(0,i)=0$ for $i \in [n]$.  Then the flows of the outgoing edges of vertices in the set $\{0,1, \ldots,i\}$ recursively determine the outgoing flow from vertex $i+1$ since $\outdeg_G(i+1)=1$. Thus, only a single integer flow $f$ is possible.
\end{proof}

We now define some notation. For a reducible graph $G_0$ on vertex set $
\{0, 1, \ldots, n, n+1\}$, let $G_1$ and $G_2$ be obtained by equations \eqref{reduction1} and \eqref{reduction2} for fixed $(i,j), (j,k) \in E(G_0).$ Let $d_i' = \indeg_{G_1}(i)-1,$ and let $d_i'' = \indeg_{G_2}(i) - 1.$ Also, let $\mathbf{d} = (0,d_2,\cdots,d_n, -\sum_{i=2}^n d_i)$ and likewise define $\mathbf{d_1} = (0,d_2',\cdots,d_n', -\sum_{i=2}^n d_i')$ and $\mathbf{d_2} = (0,d_2'',\cdots,d_n'', -\sum_{i=2}^n d_i'').$

We prove that if $G_1$ and $G_2$ satisfy Theorem~\ref{thm:vol-kpf}, so does $G_0.$ By the subdivision lemma, we have that $$\vol \mathcal F_{G_0} = \vol \mathcal F_{G_1} + \vol \mathcal F_{G_2}.$$ Hence,
it suffices we show the following lemma.

\begin{lemma}[Inductive Step] \label{inductive step}
Let $G_0,G_1,G_2$ and $\mathbf{d},\mathbf{d_1},\mathbf{d_2}$ be as defined above. Then,
\begin{equation*}
    K_{G_0}(\mathbf{d}) = K_{G_1}(\mathbf{d_1}) + K_{G_2}(\mathbf{d_2}).
\end{equation*}
\end{lemma}

\begin{proof} 
Note that since $\mathbf{d_1} \neq \mathbf{d_2},$ we have that $\mathcal F_{G_1}^\mathbb{Z}(\mathbf{d_1})$ and $\mathcal F_{G_2}^\mathbb{Z} (\mathbf{d_2})$ are disjoint.
We give a bijection \[\varphi: \mathcal F_{G_0}^\mathbb{Z}(\mathbf{d}) \to \mathcal F_{G_1}^\mathbb{Z}(\mathbf{d_1}) \;\dot\cup\; \mathcal F_{G_2}^\mathbb{Z} (\mathbf{d_2}).\] 

For an integer flow $f \in F_{G_0}^\mathbb{Z}(\mathbf{d})$, let $x = f(i,j),$ and let $y = f(j,k).$ Then we denote by $F_{G_0}^\mathbb{Z}(\mathbf{d};y\leq x)$ the subset of $F_{G_0}^\mathbb{Z}(\mathbf{d})$ where $y \leq x,$ and likewise let $F_{G_0}^\mathbb{Z}(\mathbf{d}; y>x)$ denote the the subset of $F_{G_0}^\mathbb{Z}(\mathbf{d})$ where $y > x.$ 

We  let $\varphi_1$ be restriction of $\varphi$ to $F_{G_0}^\mathbb{Z}(\mathbf{d};y\leq x)$, and $\varphi_2$ the restriction of $\varphi$ to $F_{G_0}^\mathbb{Z}(\mathbf{d};y > x)$. We now construct $\varphi_1$ and $\varphi_2$ as bijections with disjoint codomains where the union is the codomain of $\varphi$. We define $\varphi_1$ and $\varphi_2$ as illustrated in Figure~\ref{fig:recursion int flows}.

More formally, we define \[\varphi_1 : \mathcal{F}_{G_0}^\mathbb{Z}(\mathbf{d};y\leq x) \to \mathcal{F}_{G_1}^\mathbb{Z}(\mathbf{d_1}),\]
where $f \mapsto f'$ given by 
\begin{equation*}
    f'(e) = \begin{cases}
    x - y, & e = (i,j) \\
    y, & e = (i,k) \\
    f(e), & e \in E(G_1)\setminus\{(i,j),(i,k)\}.
    \end{cases}
\end{equation*}
Since the indegrees in $G_0$ and $G_1$ are the same, we see that the net flow vector is $\mathbf d_1 = \mathbf d$, so the map is well-defined. We now construct the inverse map $\varphi_1^{-1}$ with $f \mapsto f'$ given by
\begin{equation*}
    f'(e) = \begin{cases}
    f(i,j) + f(i,k), & e = (i,j) \\
    f(i,k), & e = (j,k) \\
    f(e), & e \in E(G_0)\setminus\{(i,j),(j,k)\}.
    \end{cases}
\end{equation*}
The net flow vector is again unchanged, so the map is well-defined and therefore $\varphi_1$ is a bijection.

We now construct a second bijection \[\varphi_2 : \mathcal{F}_{G_0}^\mathbb{Z}(\mathbf{d};y > x) \to \mathcal F_{G_2}^\mathbb{Z}(\mathbf{d_2}),\]
with $f \mapsto f'$ given by
\begin{equation*}
    f'(e) = \begin{cases}
    y-x-1, & e = (j,k) \\
    x, & e = (i,k) \\
    f(e), & e \in E(G_2)\setminus\{(j,k),(i,k)\}.
    \end{cases}
\end{equation*}
The only change in indegrees is that $d_j'' = d_j - 1$ and $d_k'' = d_k + 1.$ However, the outgoing flow at vertex $j$ also decreases by 1, whereas the incoming flow at vertex $k$ also decreases by 1, so the net flow vector is indeed $\mathbf d_2$. We similarly construct the inverse map $\varphi_2^{-1}$ with $f \mapsto f'$ given by
\begin{equation*}
    f'(e) = \begin{cases}
    f(i,k), & e = (i,j) \\
    f(j,k) + f(i,k) + 1, & e = (j,k) \\
    f(e), & e \in E(G_0)\setminus\{(i,j),(j,k)\}.
    \end{cases}
\end{equation*}
The only indegrees that change are $d_j = d_j'' + 1$ and $d_k = d_k'' - 1$, but since the outgoing flow at vertex $j$ increases by 1 and the incoming flow at vertex $k$ decreases by 1, we see the graph is locally unchanged. Hence, the map is well-defined and $\varphi_2$ is a bijection as a well. 

\begin{figure}
    \centering
    \includegraphics[scale=0.9]{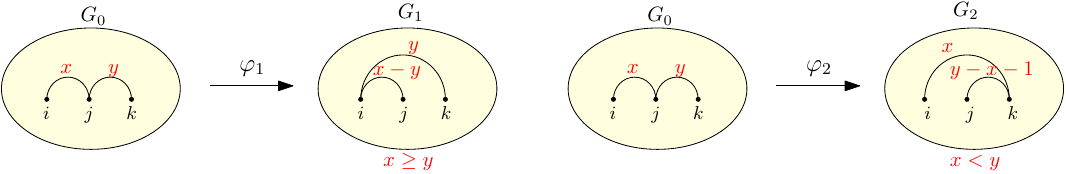}
    \caption{The maps $\varphi_1$ and $\varphi_2$ for integer flows of a subdivided graph.}
    \label{fig:recursion int flows}
\end{figure}

Since $\mathcal F_{G_1}^\mathbb{Z}(\mathbf{d_1})$ and $\mathcal F_{G_2}^\mathbb{Z} (\mathbf{d_2})$ are disjoint, we have that $\varphi$ is a bijection, and the result follows. 
\end{proof}

\begin{proof}[Recursive proof of Theorem~\ref{thm:vol-kpf}]
The proof follows from the base case given in Lemma~\ref{base case algorithm} and Lemma~\ref{base case}, and the inductive step established in Lemma~\ref{inductive step}.
\end{proof}

\section{Symmetry of \texorpdfstring{$M_n(a,b,c)$}{M_n(a,b,c)}} \label{sec: symmetry}

This section is about a fundamental symmetry of $M_n(a,b,c)$, that is invariant under switching $a$ and $b$. We give three proofs of this result with the three perspectives for $M_n(a,b,c)$ illustrated in Theorem~\ref{thm:ckm}: as a product formula, as the volume of a flow polytope, and as the number of certain integer flows. The last proof is bijective.

\begin{corollary}[Symmetry of $M_n$] \label{M symmetry}
For a positive integers $n, a$ and $b$ and nonnegative integers $c$ we have that $M_n(a,b,c)=M_n(b,a,c)$.
\end{corollary}

\subsection{Three proofs of Corollary~\ref{M symmetry}}

\begin{proof}[First proof]
This can directly be seen from symmetry of $a$ and $b$ in the product formula \eqref{morris} of $M_n(a,b,c)$.
\end{proof}

The second proof is based on Lemma~\ref{lem: reverse flow polytopes}. 

\begin{proof}
The result is a Corollary of \cite[Prop. 2.3]{MM} for netflow ${\bf a} = (1,0,\ldots,0,-1)$.
\end{proof}

\begin{proof}[Second proof of Corollary \ref{M symmetry}]
Since the reverse of the graph $k_{n+2}^{a,b,c}$ is the graph $k_{n+2}^{b,a,c}$ then by Lemma~\ref{lem: reverse flow polytopes} we have that   
$\mathcal{F}_{k_{n+2}^{a,b,c}} \equiv  \mathcal{F}_{k_{n+2}^{b,a,c}}$. The volume of the two polytopes must be equal and the result follows by \eqref{eq: Mn is kpf and volume}.
\end{proof}

The third proof is based on the following lemma from \cite{MM}.  The lemma was originally proved geometrically by combining Lemma~\ref{lem: reverse flow polytopes} with Theorem~\ref{thm:vol-kpf}. We give here a bijective proof of this result.

\begin{lemma}[{\cite[Cor. 1.4]{MM}}] \label{lem: reverse lattice points}
For a loopless digraph $G$ with vertices $\{0,1,\ldots,n+1\}$ having a unique source $0$ and unique sink $n+1$ we have
\[
K_G(0,d_1,\ldots,d_n,-\sum_{i=1}^n d_i) = K_{G^r}(0,d^r_1,\ldots,d^r_n,-\sum_{i=1}^n d^r_i),
\]
where $d_i = \indeg_G(i)-1$ and $d^r_i=\indeg_{G^r}(i)-1$.
\end{lemma}

\begin{remark}
Lemma~\ref{lem: reverse lattice points} states that two flow polytopes have the same number of lattice points. These polytopes can have different dimensions because of the different vertices with zero netflow. That is, if $G|_{[n+1]}$ is the subgraph of $G$ restricted to vertices $[n+1]$ then the  equivalent polytopes $\mathcal{F}_{G|_{[n+1]}}(d_1,\ldots,d_n,-\sum_i d_i)$ and $\mathcal{F}_{G^r|_{[n+1]}}(d^r_1,\ldots,d^r_n,-\sum_i d^r_i)$ have dimensions $E(G|_{[n+1]})-n-2$ and $E(G^r|_{[n+1]})-n-2$, respectively.
\end{remark}

\begin{proof}[Third proof of Corollary \ref{M symmetry}]
Since the reverse of the graph $k_{n+2}^{a,b,c}$ is the graph $k_{n+2}^{b,a,c}$ then by Lemma~\ref{lem: reverse lattice points} we have that 
\[
K_{k_{n+2}^{a,b,c}}(0,t_1,\ldots,t_n,-\sum_{i=1}^n t_i) \,=\, K_{k_{n+2}^{b,a,c}}(0,s_1,\ldots,s_n,-\sum_{i=1}^n s_i),
\]
where $t_i=a-1+c(i-1)$ and  $s_i=b-1+c(i-1)$ for $i=1,\ldots,n$. The result then follows by Theorem~\ref{thm:ckm}.
\end{proof}

The rest of the section is devoted to the bijective proof of Lemma~\ref{lem: reverse lattice points} using a bijection from \cite[Section 7]{MMS} and inspired by a bijection between lattice points of generalized permutahedra (see \cite[Thm. 12.9]{Pos} and \cite{GNP}). 

\subsection{Bijection between lattice points of \texorpdfstring{$\mathcal{F}_G(0,d_1,\ldots)$}{F_G(0,d1,...)} and \texorpdfstring{$\mathcal{F}_{G^r}(0,d^r_1,\ldots)$}{F_(G^r)(0,d^r_1,...)}}

First, we describe the Danilov, Karzanov, Koshevoy   (DKK) triangulations of $\mathcal{F}_G$ from \cite{DKK}. Recall that the vertices of $\mathcal{F}_G$ are given by unit flows along routes on $G$, i.e. directed paths in $G$ from the source $0$ to the sink $n+1$. Given a route $R$ with vertex $v$, $Rv$ and $vR$ denote the subpaths ending and starting at $v$, respectively. A \textit{framing} at an inner vertex $v$ is a pair ($\preceq_{in(i)}$, $\preceq_{out(i)})$ of linear orderings on the set of incoming edges $in(v)$ to $v$ and on the set of outgoing edges $out(v)$ from $v$. A \textit{framed graph} $(G,\preceq)$ is a graph $G$ with a framing $\preceq$ at each internal vertex.

Fix a framed graph $(G,\preceq)$. For an internal vertex $i$, let $In(i)$ and $Out(i)$ be the sets of maximal paths ending and starting at $i$, respectively. Given a framed graph, we define an order $\preceq_{In(i)}$ on $In(i)$ as follows. Given distinct paths $R,Q$ in $In(i)$, let $j\leq i$ be the smallest vertex after which $Ri$ and $Qi$ coincide and let $e_R$ be the edge of $R$ entering $i$ and $e_Q$ be the edge of $Q$ entering $i$. We have $R \preceq_{In(i)} Q$ if and only if $e_R \preceq_{in(i)} e_Q$. We define an order $\preceq_{Out(i)}$ on $Out(i)$ analogously. We say that routes $R$ and $Q$ with a common inner vertex $i$ are \textit{coherent at $i$} whenever $Ri \preceq_{In(i)} Q_i$ if and only if $iR \preceq_{Out(i)} iQ$. That is, if the paths $iR$ and $iQ$ are ordered the same as $Ri$ and $Qi$. Routes $R$ and $Q$ are \textit{coherent} if they are coherent at each internal vertex they have in common. A set of pairwise coherent routes is called a \textit{clique}. Let $\mathcal{C}^{\max}(G,\preceq)$ be the set of maximal cliques of $(G,\preceq)$. For a maximal clique $C$, let $\Delta_C$ be the convex hull of the vertices of $\mathcal{F}_G$ corresponding to the routes in $C$.

In \cite[Thm. 1. \& 2]{DKK} show that given a framed graph $(G,\preceq)$, the set
$\{\Delta_C \mid C \in \mathcal{C}^{\max}(G,\preceq)\}$ is the set of top dimensional simplices in a regular unimodular triangulation of $\mathcal{F}_G$. See Figure~\ref{fig:DKK triangulation}. As a corollary, we obtain another combinatorial object, maximal cliques, whose number gives the volume of $\mathcal{F}_G$.

\begin{corollary}
For a framed graph $(G,\preceq)$ where $G$ is a loopless digraph with vertices $\{0,1,\ldots,n+1\}$ having a unique source $0$ and sink $n+1$ then $K_G(0,d_1,\ldots,d_n,-\sum_{i=1}^n d_i) = |\mathcal{C}^{\max}(G,\preceq)|$.
\end{corollary}

\begin{figure}
    \centering
    \includegraphics[scale=0.7]{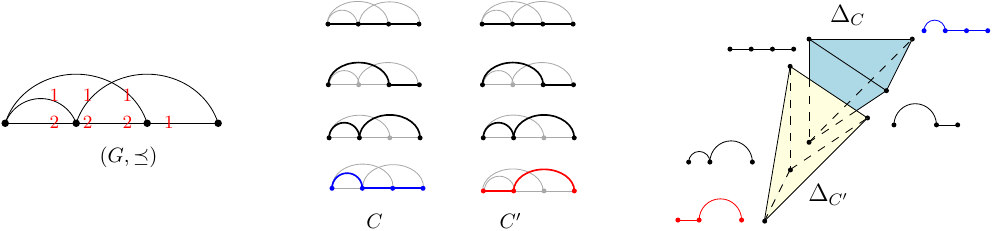}
    \caption{A framed graph $(G,\preceq)$, the two maximal cliques in $\mathcal{C}^{\max}_{G,\preceq}$ and the corresponding DKK triangulation of $\mathcal{F}_G$.}
    \label{fig:DKK triangulation}
\end{figure}

In \cite{MMS} the authors give an explicit bijection between the maximal cliques and the integer flows. The map from the cliques to the flows is as follows.

\begin{definition}[{\cite[Section 7]{MMS}}]
Given a framed graph $(G,\preceq)$ and ${\bf d}=(0,d_1,\ldots,d_n,-\sum_i d_i)$, let $\Omega_{G,\preceq}: \mathcal{C}^{\max}(G,\preceq) \to \mathcal{F}_G^{\mathbb{Z}}({\bf d})$ be defined as follows $\Omega^{(G,\preceq)}: C \mapsto f$ where $f(e)=n(e)-1$ where $n(i,j)$ is the number of times edge $(i,j)$ appears in the set of prefixes $\{Rj \mid R \in C\}$.
\end{definition}

\begin{lemma}[{\cite[Lemma 7.9]{MMS}}] \label{bij: clique to flow}
Given a framed graph $(G,\preceq)$, the map $\Omega_{G,\preceq}$ is a bijection between maximal cliques $\mathcal{C}^{\max}(G,\preceq)$ and integer flows in $\mathcal{F}^{\mathbb{Z}}(0,d_1,\ldots,d_n,-\sum_i d_i)$.
\end{lemma}

The inverse map $\Omega^{-1}_{G,\preceq}$ can be found in \cite[Section 7]{MMS}) (denoted $\Lambda_{G,\preceq}$). 

A framing $(G,\preceq)$ induces the following framing on the reverse graph $G^r$: for an internal vertex $i$, let $\preceq^r_{in(i)} :=\, \preceq_{out(i)}$ and $\preceq^r_{out(i)} :=\, \preceq_{in(i)}$. We denote this induced framing by $(G^r,\preceq^r)$.

\begin{remark} \label{rmk: reversing cliques}
Note that the framings $(G,\preceq)$ and $(G^r,\preceq^r)$ induce the same triangulation of $\mathcal{F}_G \equiv \mathcal{F}_{G^r}$. In other words up to reversing the routes, the cliques of $(G,\preceq)$ and $(G^r,\preceq^r)$ are the same. Thus $r:C \mapsto \{ Q^r \mid Q \in C\}$ is a bijection between $\mathcal{C}^{\max}(G,\preceq)$ and $\mathcal{C}^{\max}(G^r,\preceq^r)$.
\end{remark}

We are now ready to define the  map that will give the desired bijection from the integer flows in $G$ and the integer flows in $G^r$.

\begin{definition} \label{def: master bijection}
Given a framed graph $(G,\preceq)$, let $\Theta_{G,\preceq} := \Omega_{G^r,\preceq^r} \circ r\circ \Omega^{-1}_{G,\preceq}$ which is a map from $\mathcal{F}_G^{\mathbb{Z}}(0,d_1,\ldots,d_n,-\sum_i d_i)$ to $\mathcal{F}_{G^r}^{\mathbb{Z}}(0,d^r_1,\ldots,d^r_n,-\sum_i d^r_i)$ . See Figure~\ref{fig:bijection flows} for example.
\end{definition}

\begin{lemma} \label{lemma: the bijection}
Given a framed graph $(G,\preceq)$, the map $\Theta_{G,\preceq}$ is a bijection between $\mathcal{F}_G^{\mathbb{Z}}(0,d_1,\ldots,d_n,-\sum_i d_i)$ to $\mathcal{F}_{G^r}^{\mathbb{Z}}(0,d^r_1,\ldots,d^r_n,-\sum_i d^r_i)$ .
\end{lemma}

\begin{proof}
By Lemma~\ref{bij: clique to flow} $\Omega^{-1}_{G,\preceq}$ is a bijection between $\mathcal{F}^{\mathbb{Z}}(0,d_1,\ldots,d_n,-\sum_i d_i)$ and $\mathcal{C}^{\max}(G,\preceq)$. Next, by reversing the routes in the cliques, we identify $C$ in $\mathcal{C}^{\max}(G,\preceq)$ with $r(C)$ in  $\mathcal{C}^{\max}(G^r,\preceq^r)$. Finally, by Lemma~\ref{bij: clique to flow} the latter is in bijection with $\mathcal{F}_{G^r}^{\mathbb{Z}}(0,d^r_1,\ldots,d^r_n,-\sum_i d^r_i)$ via $\Omega_{G^r,\preceq^r}$.
\end{proof}

Finally, the correspondence $\Theta_{G,\preceq}$ gives the bijective proof of Lemma~\ref{lem: reverse lattice points}.

\begin{proof}[Proof of Lemma \ref{lem: reverse lattice points}]
The result follows by Lemma~\ref{lemma: the bijection}.
\end{proof}

\begin{remark}
The DKK triangulation of a flow polytope $\mathcal{F}_G$ gives the bijection between the integer flows. Indeed, given a clique in $\mathcal{C}^{\max}_{G,\preceq}$, $\Omega_{G,\preceq}:C\mapsto f$ where $f(i,j)=n(i,j)-1$ where $n(i,j)$ is the number of times $(i,j)$ appears in the set of prefixes $\{jR\mid R \in C\}$ and $\Omega_{G^r,\preceq^r}:C \mapsto f'$ where $f'(i,j)=n'(i,j)-1$, where $n'(i,j)$ is the number of times edge $(n+1-j,n+1-i)$ appears in the set of  suffixes $\{(n+1-j)R \mid R \in C\}$. See Figure~\ref{fig:bijection flows}.
\end{remark}

\begin{figure}
    \centering
    \includegraphics[scale=0.7]{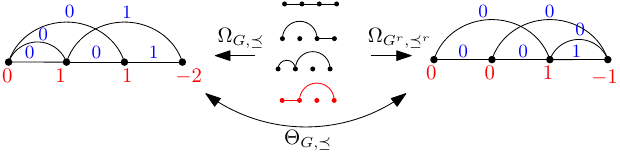} \qquad \includegraphics[scale=0.7]{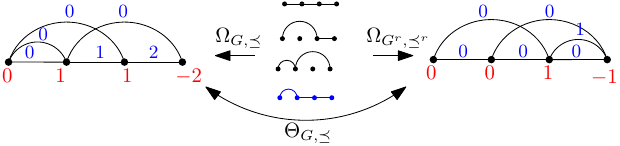}
    \caption{Example of the correspondence $\Theta_{G,\preceq}$  between integer flows in $\mathcal{F}_G(0,d_1,d_2,\ldots)$ and $\mathcal{F}_{G^t}(0,d^r_1,d^r_2,\ldots)$.}
    \label{fig:bijection flows}
\end{figure}

\begin{example} \label{ex: prod simplices}
For positive integers $p$ and $q$, let  $G(p,q)$ be the graph  with vertices $\{0,1,2\}$, $p$ edges $(0,1)$, and $q$ edges $(1,2)$. One readily sees that the flow polytope $\mathcal{F}_{G(p,q)}$ is the product of simplices $\Delta^{p-1}\times \Delta^{q-1}$. By Theorem~\ref{thm:BV} we have that 
\[
\vol \mathcal{F}_{G(p,q)} = K_{G(p,q)}(0,p-1,-p+1) = \#( (p-1)\Delta^{q-1} \cap \mathbb{Z}^q) = \binom{p+q-2}{p-1}. 
\]
Since $G(p,q)^r = G(q,p)$, then for each of the $p!q!$ framings of $G(p,q)$ we have that   $\Theta^{G(p,q),\preceq}$ is a bijection between the lattice points of $(p-1)\Delta^{q-1}$ and $(q-1)\Delta^{p-1}$. By a result of Postnikov \cite[Lemma 12.5]{Pos}, if $C$ is a maximal clique in $\mathcal{C}^{\max}_{G(p,q)}$, then the following subgraph of the complete bipartite graph on vertices $\{1,\ldots,p\}\cup \{\overline{1},\ldots,\overline{q}\}$ is a spanning tree: the subgraph $T$ with edges $(i,\overline{j})$ if $C$ contains the route of $G(p,q)$ with the $i$th edge $(0,1)$ and $j$th edge $(1,2)$ in the framing  (see Figure~\ref{fig:ex prod simplices}). Moreover, the coherent condition on the routes of the clique translate to the {\em compatiblity} condition on spanning trees (see \cite[Section 12]{Pos} and \cite[Section 2]{GNP}), and the bijection $\Theta^{G(p,q),\preceq}$ consists of mapping the right degrees minus one to the left degrees minus one of the spanning tree $T$ \cite[Thm. 12.9]{Pos}. 
\end{example}

\begin{figure}
    \centering
    \includegraphics[scale=0.8]{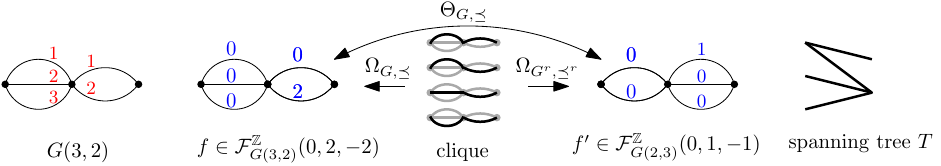}
    \caption{Example of correspondence  $\Theta_{G,\preceq}$ that gives the same bijection between the lattice points of $(p-1)\Delta^{q-1}$ and $(q-1)\Delta^{p-1}$ studied by Postnikov \cite{Pos} in terms of spanning trees.}
    \label{fig:ex prod simplices}
\end{figure}

\section{A new refinement of \texorpdfstring{$M_n(a,b,c)$}{M_n(a,b,c)}} \label{sec: psi}

We define the following constant term in order to refine $M_n(a,b,c)$.

\begin{definition}
Define the following constant term
\begin{equation*}
\Psi_n(k,a,b,c) := \CT_x[t^k]\prod_{i=1}^n (1-x_i)^{-b}x_i^{-a+1}\left(1+t\frac{x_i}{1-x_i}\right) \prod_{1 \leq i < j \leq n} (x_j - x_i)^{-c}.
\end{equation*}
\end{definition}

\subsection{Volume and Kostant partition function interpretations for \texorpdfstring{$\Psi_n(k,a,b,c)$}{Psi_n(k,a,b,c)}} We prove both parts of Theorem \ref{psi kpf}. 
\begin{proof}[Proof of Theorem~\ref{psi kpf}~(i)]
We first prove the Kostant partition interpretation of $\Psi_n(k,a,b,c)$. We specialize the generating function in equation \eqref{eq: kpf gen function kabc} as in the proof of Theorem~\ref{thm:ckm} in Section~\ref{sec: kpf}. Note that compared with $M_n(a,b,c)$, $\Psi_n(k,a,b,c)$ has an extra term $[t^k]\prod_{i=1}^n(1+t\frac{x_i}{1-x_i})$. This term selects $k$ values of $\{1,\ldots,n\}$ and for each selected $i$ multiplies the generating series by $\frac{x_i}{1-x_i}.$

By linearity of constant term extraction,
\begin{equation} \label{eq: CT to coeff extraction}
    \CT_{x_i} \frac{x_i}{1-x_i} f(x_1, \ldots, x_n)  = \sum_{r=1}^{\infty} \CT_{x_i} x_i^r f(x_1, \ldots, x_n) = \sum_{r=1}^{\infty} [x_i^{-r}] f(x_1, \ldots, x_n).
\end{equation}

We then substitute the generating function for $M_n(a,b,c)$ is substituted for $f(x_1,\ldots,x_n)$ in the RHS of \eqref{eq: CT to coeff extraction} and apply the Kostant partition function interpretation of $M_n(a,b,c)$. However, instead of taking the case where the net flow at vertex $i$ is $a-1 + c(i-1),$ we sum the cases where the net flow at $i$ is $a-2 + c(i-1), a-3 + c(i-1), \ldots$ Hence, this is equivalent to strictly decreasing the net flow at vertex $i$ in the Kostant partition function interpretation.  Since we take the coefficient of $t^k,$ there are exactly $k$ vertices with net flow $a_i < a-1 + c(i-1),$ and $n-k$ vertices with net flow $a_i = a-1+c(i-1).$ The result follows.

\end{proof}

\begin{example} The number $\Psi_2(k,a,b,c)$ for $k=0,1,2$ counts the following sums of Kostant partition functions:
\begin{align*}
\Psi_2(0,a,b,c) &= K_{k_4^{a,b,c}}(0,a-1,a+c-1)\\
\Psi_2(1,a,b,c) &=  \sum_{t_2<a+c-1} K_{k_4^{a,b,c}}(0,a-1,t_2,-a-t_2+1) +\sum_{t_1<a-1}  K_{k_4^{a,b,c}}(0,t_1,a+c-1,-t_1-a-c+1) \\
\Psi_2(2,a,b,c) &= \sum_{t_1<a-1,\;t_2<a+c-1} K_{k_4^{a,b,c}}(0,t_1,t_2,-t_1-t_2).\\
\end{align*}
\end{example}

We now prove the volume interpretation. To do so, we first define a modification of $k_{n+2}^{a,b,c}.$

\begin{definition}
For a set $S\subseteq [n]$, let $k^{a,b,c}_{n+2}(S)$ be the graph obtained from $k^{a,b,c}_{n+2}$ by adding $n$ edges $(0,n+1)$ and for each $i\in S$ we delete one of the $a$ incoming edges $(0,i)$ and add an outgoing edge $(i,n+1)$.
\end{definition}

For a set $S \subseteq [n],$ define also the set $T(S)$ as the set of vectors $\mathbf{a}=(0,a_1,a_2, \ldots, a_n, -\sum_{j=1}^n a_j)$ with $a_i < a-1 + c(i-1)$ for $i \in S$ and $a_i = a-1 + c(i-1)$ for $i \notin S.$

\begin{proof}[Proof of Theorem~\ref{psi kpf}~(ii)]
First we show that \begin{equation} \label{T(S) proof}
    \sum_{\mathbf{a} \in T(S)} K_{k_{n+2}^{a,b,c}}(\mathbf{a}) = \vol \mathcal{F}_{k^{a,b,c}_{n+2}(S)}.
\end{equation}

Consider the Kostant partition functions on the left-hand side. For each vertex $i \in S,$ we remove an incoming edge $(0,i)$ (decreasing $a$ by 1) to create a weak inequality instead of a strict equality. We then add an outgoing edge $(i,n+1)$ to carry the necessary flow to force the equality $a_i = \indeg(i)-1$. This process is shown in Figure~\ref{fig:adding edges}.

We note that if we were to add the edge $(0,n+1)$ $n$ times, the volume would not change. This is due to Theorem~\ref{thm:vol-kpf}, which gives the volume as a Kostant partition function where the source has zero net flow. Since an edge $(0,n+1)$ also would not affect the indegree of any internal vertex, it has no effect on the Kostant partition function or volume. By adding the edge $(0,n+1)$ $n$ times, the graph becomes $k{^{a,b,c}_{n+2}(S)}$. 

Hence, since each internal vertex has net flow $a_i = \indeg(i)-1$, we apply Theorem~\ref{thm:vol-kpf} again to obtain that this Kostant partition function is equal to $\vol \mathcal{F}_{k^{a,b,c}_{n+2}(S)},$ thus proving equation \eqref{T(S) proof}. 

We can now sum both sides of equation \eqref{T(S) proof} over $S \in \binom{[n]}{k}$. Since the Kostant partition function interpretation of $\Psi_n(k,a,b,c)$ counts all flows with exactly $k$ strict inequalities $a_i < a-1 + c(i-1),$ we see that the left-hand side is now $\Psi_n(k,a,b,c)$, and the result follows. 
\end{proof}

On a polytope level, the volume interpretation of Theorem~\ref{psi kpf}~(ii)translates to the following result.

\begin{lemma} \label{psi polytope} For $S \subseteq [n],$ the polytopes $\mathcal{F}_{k^{a,b,c}_{n+2}(S)}$ are interior disjoint and satisfy 
\[
\mathcal{F}_{k^{a,b+1,c}_{n+2}} \equiv \bigcup_{S \subseteq [n]} \mathcal{F}_{k^{a,b,c}_{n+2}(S)}.
\]
\end{lemma}
\begin{proof}
We apply the subdivision lemma \eqref{reduction1} and \eqref{reduction2} at each internal vertex of $F_{k^{a,b+1,c}_{n+2}}$ exactly once. For each internal vertex $i$, the edge $(0,n+1)$ is added, and either an incoming edge $(0,i)$ or outgoing edge $(i,n+1)$ is deleted. This is shown in Figure~\ref{fig:pf volume subdiv}.

\begin{figure}
     \centering
     \begin{subfigure}[b]{0.45\textwidth}
     \raisebox{20pt}{\includegraphics[width=\textwidth]{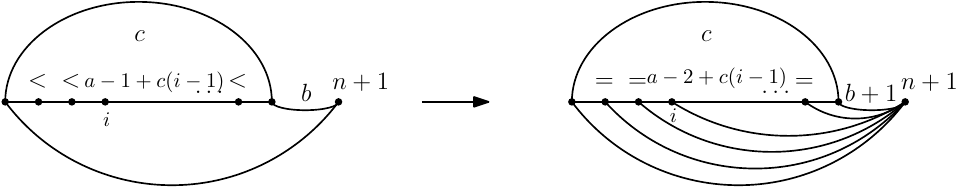}}
    \caption{} 
    \label{fig:adding edges}
     \end{subfigure}
     \begin{subfigure}[b]{0.45\textwidth}
         \centering
         \includegraphics[width=\textwidth]{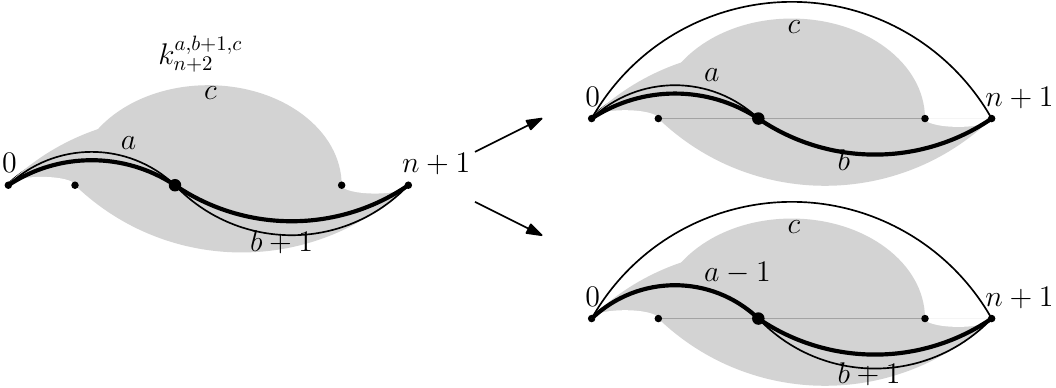}
         \caption{}
         \label{fig:pf volume subdiv}
     \end{subfigure}
     \caption{(A) For each strict inequality where $a_i < \indeg(i) - 1,$ the inequality can be weakened by decreasing $a$ by 1 to obtain $a_i \leq \indeg(i)-1.$ We then add an additional edge $(i, n+1)$ to carry the necessary flow such that $a_i = \indeg(i) - 1.$. (B)  The graph $k_{n+2}^{a,b+1,c}$ is shown in gray on the left with some highlighted edges in black. The graphs on the right give the two subdivisions obtained from applying the subdivision lemma on a single internal vertex.}
\end{figure}

That is, for each vertex $i,$ one of two cases must hold:
\begin{enumerate}[label=(\roman*)]
    \item Edge $(0, i)$ appears $a-1$ times and edge $(i,n+1)$ appears $b+1$ times.
    \item Edge $(0,i)$ appears $a$ times and edge $(i,n+1)$ $b$ times. 
\end{enumerate}
Each reduced graph is the polytope $\mathcal{F}_{k^{a,b,c}_{n+2}(S)},$ where $S$ is the set of vertices satisfying case (i). Since the polytopes $\mathcal{F}_{k^{a,b,c}_{n+2}(S)}$ are obtained by applying the reduction rule on $\mathcal{F}_{k^{a,b+1,c}_{n+2}}$,  they are interior disjoint by Proposition~\ref{reduction proposition}, and their union is integrally equivalent to $\mathcal{F}_{k^{a,b+1,c}_{n+2}}$.
\end{proof}

As an application of these interpretations we now prove Corollary~\ref{cor:M refinement}, which refines the product $M_n(a,b,c)$.

\begin{example} The polytope $\mathcal F_{k_4^{a,b+1,c}}(1,0,\dots,0,-1)$ can be subdivided into four flow polytopes of the form $\mathcal F_{k_4^{a,b+1,c}(S)}(1,0,\dots,0,-1)$, which are grouped into three collections based on the size of $S$. The volume of the collection corresponding to $|S|=k$ is counted by $\Psi_2(k,a,b,c)$. See Figure~\ref{fig: polytope interpretation Psi}.
\end{example}

\begin{figure}
    \centering
   \includegraphics[scale=0.34]{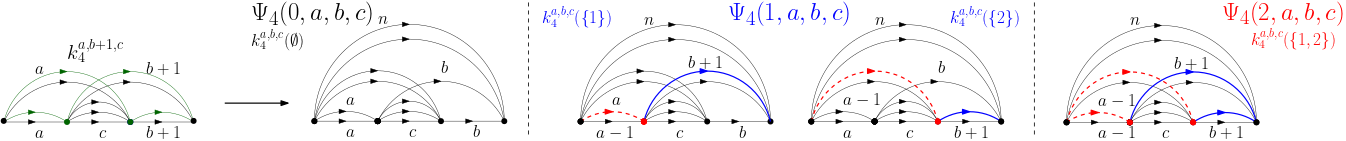}
    \caption{Three collections of flow polytopes $\mathcal{F}_{k_4^{a,b,c}}(S)$ of the subdivision of $\mathcal{F}_{k_4^{a,b,c}}$ whose volumes are given by $\Psi_2(k,a,b,c)$ for $k=0,1,2$ respectively.}
    \label{fig: polytope interpretation Psi}
\end{figure}

\begin{proof}[Proof of Corollary~\ref{cor:M refinement} via Kostant partition function]
The sum on the right-hand side of \eqref{eq: Psi refine M} over $k$ of $\Psi_n(k,a,b,c)$ is the sum of all Kostant partition functions $K_{k_{n+2}^{a,b,c}}(0, a_1, \dots, a_{n}, -\sum_{j=1}^n a_j)$ such that for $i \in [n],$ $a_i \leq a - 1 + c(i-1).$ This is equivalent to adding another edge between each vertex $i$ and the sink with flows such that each net flow satisfies $a_i = a-1 + c(i-1)$. Hence we see this sum is $M_n(a,b+1,c).$
\end{proof}

\begin{proof}[Proof of Corollary~\ref{cor:M refinement} via volumes]
Using Lemma~\ref{psi polytope} and computing the volume on both sides gives 
\[
\vol \mathcal{F}_{k^{a,b+1,c}_{n+2}} = \sum_{S \subseteq [n]} \vol \mathcal{F}_{k^{a,b,c}_{n+2}(S)}.
\]
The result follows by applying Theorem~\ref{thm:ckm} to the left-hand side and Theorem~\ref{psi kpf} to the right-hand side.
\end{proof}

\subsection{Recurrence Relations of \texorpdfstring{$\Psi_n(k,a,b,c)$}{Psi_n(k,a,b,c)}}

In this section we prove recurrence relations satisfied by $\Psi_n(k,a,b,c)$ and that are instrumental to our proof of Theorem~\ref{thm:psi product}. First we show two cases where $\Psi_n(k,a,b,c)$ is equivalent to the Morris identity.

\begin{proposition} \label{phi m}
Let $n, a,b,c$ be positive integers. Then
\begin{align}
    \Psi_n(0,a,b,c) &= M_n(a,b,c)\\
    \Psi_n(n,a,b,c) &= M_n(a-1,b+1,c).
\end{align}
\end{proposition}

\begin{proof}
The first equation holds since $n$ equalities implies the exact same constraints as those of $M_n(a,b,c).$
The second equation holds since 0 equalities implies the upper bound of the inequalities can be decreased by 1 (by decreasing $a$ by 1) to make a weak inequality, and another edge from each vertex to the sink can be added with the necessary flow to force equality. This transformation gives a bijection with $M_n(a-1,b+1,c).$
\end{proof}

We now construct a bijection on the sets of integer flows of flow polytopes $F_{k_{n+2}^{1,b,c}}(0, 0, a_2, \dots, a_{n}, -\sum_{j=2}^n a_j)$ and $\mathcal F_{k_{n+1}^{c+1,b,c}}(0, a_2, \dots, a_{n}, -\sum_{j=2}^n a_j)$ by contracting an edge in the corresponding graph. 

\begin{proposition}
For a net flow vector $\mathbf{a} = (0, 0, a_2, \dots, a_{n}, -\sum_{j=2}^n a_j),$ let $\mathbf{\tilde a} := (0,a_2, \dots, a_{n}, -\sum_{j=2}^n a_j)$. Then for positive integers $b$ and $n$ and nonnegative integer $c,$
\begin{equation} \label{varphi kpf} 
    K_{k_{n+2}^{1,b,c}}(\mathbf{a}) = K_{k_{n+1}^{c+1,b,c}}(\mathbf{\tilde a}).
\end{equation}
\end{proposition}

\begin{proof}
Define the map
\[
\varphi: \mathcal F^\mathbb{Z}_{k_{n+2}^{1,b,c}}(\mathbf{a}) \to \mathcal F^\mathbb{Z}_{k_{n+1}^{c+1,b,c}}(\mathbf{\tilde a}),
\]

where $f \mapsto f',$ with $f'(i,j)$ given by \begin{equation*}
    f'(i,j) = \begin{cases}
    0, & i = 0 \\
    f(i+1,j+1), & 1 \leq i < j \leq n.
        \end{cases}
\end{equation*}

This map contracts edge $(0,1)$ to create vertex $0$ and relabels each vertex $i \in [2,n+1]$ by $i-1.$ We see the $c$ edges $(1, i)$ for $i \in [2, n]$ become identical with the edges of the form $(0, i), i \in [2,n].$ Hence, the graph transforms to become $k_{n+1}^{c+1,b,c}.$ We now show that $\varphi$ is a bijection. It is sufficient we show $\varphi$ has a well-defined inverse function for all $f \in \mathcal F_{k_{n+1}^{c+1,b,c}}(\mathbf{\tilde a})$. Define the inverse map
\[
\varphi^{-1}: \mathcal F^\mathbb{Z}_{k_{n+1}^{c+1,b,c}}(\mathbf{\tilde a}) \to \mathcal F^\mathbb{Z}_{k_{n+2}^{1,b,c}}(\mathbf{a})
\]

where $f \mapsto f',$ with $f'(i,j)$ given by \begin{equation*}
    f'(i,j) = \begin{cases}
    0, & 0 \leq i \leq 1 \\
    f(i-1,j-1), & 2 \leq i < j \leq n+1.
    \end{cases}
\end{equation*}

Note that for $(i,j) \in E(\mathcal F^\mathbb{Z}_{k_{n+2}^{1,b,c}}(\mathbf{a})),$ $f(0,j) = f(1,j) = 0$ as the net flows at vertices 0 and 1 are both zero, so we see $\varphi^{-1}$ is indeed our desired inverse function. Thus $\varphi$ is a bijection, and the result follows.
\end{proof}

We further strengthen this contraction identity to hold bijectively for $\Psi_n(k,1,b,c).$

\begin{lemma}[Contraction Lemma] \label{psi contraction} For positive integers $b$ and $n$ and nonnegative integers $c$ and $k \leq n,$ 
\begin{equation*} 
    \Psi_n(k,1,b,c) = \Psi_{n-1}(k,c+1,b,c).
\end{equation*}
\end{lemma}
\begin{proof}
Recall $\Psi_n(k,1,b,c)$ is the sum of Kostant partition functions of the form $K_{k_{n+2}^{1,b,c}}(\mathbf{a}),$ with $\mathbf{a} = (0, 0, a_2, \dots, -\sum_{j=2}^n a_j)$ where for $i \in [2,n],$ $a_i < c(i-1)$ holds $k$ times and $a_i = c(i-1)$ holds $n-1-k$ times (since the first internal vertex trivially satisfies this equality). Let $A$ be the set of all such $\mathbf{a}$ satisfying these conditions. That is,  \begin{equation*}
    \Psi_n(k,1,b,c) = \sum_{\mathbf{a} \in A} K_{k_{n+2}^{1,b,c}}(\mathbf{a}).
\end{equation*}

Similarly, let $A'$ be the set of all $\mathbf{a'} = (0, a_2, \dots, -\sum_{j=2}^n a_j)$ where for $i \in [2,n],$ $a_i < c(i-1)$ holds $k$ times and $a_i = c(i-1)$ holds $n-1-k$ times. Then \begin{equation*}
    \Psi_{n-1}(k,c+1,b,c) = \sum_{\mathbf{a'} \in A'} K_{k_{n+2}^{c+1,b,c}}(\mathbf{a'}).
\end{equation*}

We see that the map $\varphi: A \to A', (0, 0, a_2, \dots, -\sum_{j=2}^n a_j) \mapsto (0, a_2, \dots, -\sum_{j=2}^n a_j)$ is a bijection. By equating the Kostant partition functions using equation \eqref{varphi kpf}, the result follows.
\end{proof}

As a result of this lemma, the following two corollaries are immediate.

\begin{corollary} \label{M contraction} For positive integers $b$ and $n$ and nonnegative integer $c$,
\begin{equation*}
    M_n(1,b,c) = M_{n-1}(c+1,b,c).
\end{equation*}
\end{corollary}
\begin{proof}
This is a result of Lemma \ref{psi contraction} when $k = 0.$
\end{proof}

\begin{corollary} \label{a11 refinement}
For positive integers $a$ and $n$, it holds that
\begin{equation*}
M_n(a,1,1) = \sum_{k=0}^n \Psi_{n-1}(k,a,1,1).
\end{equation*}
\end{corollary}

\begin{proof}
By Corollary \ref{M contraction} and Proposition \ref{M symmetry}, we see that $M_n(a,1,1) = M_{n-1}(a,2,1).$ Hence applying Lemma \ref{cor:M refinement}, the result follows.
\end{proof}

Following the approach of Baldoni-Vergne in their proof of Theorem~\ref{thm:BV}, we now give relations of $\Psi_n(k,a,b,c)$ that we later show uniquely determine this function.

\begin{lemma} \label{psi recurrences}
For nonnegative integer $c,$ positive integers $a,b,n,$ and nonnegative integer $k \leq n,$ the constant term $\Psi_n(k,a,b,c)$ satisfies the following identities:
\begin{align}
      \Psi_n(n,a,b,c) &= \Psi_n(0,a-1,b+1,c) \label{psi rel 1} \\
      \Psi_n(n-1,1,b,c) &= \Psi_{n-1}(0,c,b+1,c)  \label{psi rel 2}\\
     \Psi_n(0,1,b,0) &= 1 \label{psi rel 3}\\
     k(b+(k-1)c/2)\cdot \Psi_{n}(k,a,b,c) &= (n-k+1)(a-1+(n-k)c/2) \cdot \Psi_{n}(k-1,a,b,c)  \textrm{ for } 1 \leq k \leq n \label{psi rel 4}.
\end{align}
\end{lemma}

\begin{proof}
The first relation follows immediately from the bijections in Proposition \ref{phi m}. The second relation follows by applying Lemma \ref{psi contraction} to the left-hand side, which turns the equation into $\Psi_{n-1}(n-1,c+1,b,c) = \Psi_{n-1}(0,c,b+1,c),$ which follows from the first relation. The third relation follows from applying the interpretation of Theorem~\ref{psi kpf}~(i) to the left-hand side of \eqref{psi rel 3}, which gives $a_i = a - 1 +c(n-1) = 0$ for all $i \in [n].$ Hence there is only one flow where the flow along each edge is zero. Lastly we prove the fourth relation. This is the sole relation we prove algebraically. 

Let $U := \prod_{i=1}^n (1-x_i)^{-b}x_i^{-a} \prod_{1 \leq i < j \leq n} (x_j-x_i)^{-c},$ and let $P_k := k!(n-k)![t^k]\prod_{i=1}^k (1+t\frac{x_i}{1-x_i}),$ where $[t^k]\prod_{i=1}^k (1+t\frac{x_i}{1-x_i}) = e_k(\frac{x_i}{1-x_i})$ is the $k$th elementary symmetric function in $\frac{x_i}{1-x_i}.$ We have that
\begin{align}
    & \frac{\partial}{\partial x_1}(1-x_1)\frac{x_1}{1-x_1}\cdots\frac{x_k}{1-x_k}U \\
    & = \left(b\frac{x_1}{1-x_1}\cdots\frac{x_k}{1-x_k} +(1-a)\frac{x_2}{1-x_2}\cdots\frac{x_k}{1-x_k}+c(1-x_1)\frac{x_1}{1-x_1}\cdots\frac{x_k}{1-x_k}\sum_{j=2}^n \frac{1}{x_j-x_1}\right)U. \notag
\end{align}

If $c$ is odd, then $U$ is antisymmetric. Anti-symmetrizing over $\mathfrak S_n$ gives:
\begin{align}
    & \sum_{w \in \mathfrak S_n} (-1)^w w \cdot \left(\frac{\partial}{\partial x_1}(1-x_1)\frac{x_1}{1-x_1}\frac{x_2}{1-x_2}\cdots\frac{x_k}{1-x_k}U\right) \\  
    \label{antisymmetrizing} & = bP_kU + (1-a)P_{k-1}U + c \sum_{w \in \mathfrak S_n} w \cdot \left((1-x_1)\frac{x_1}{1-x_1}\frac{x_2}{1-x_2}\cdots\frac{x_k}{1-x_k}\sum_{j=2}^n \frac{1}{x_j-x_1}\right)U.
\end{align}
To evaluate the sum, we seek pairings of summands that reduce easily. Consider when $w$ is the identity permutation. Then for each summand in $\sum_{j=2}^n \frac{1}{x_1-x_j}$, for $2 \leq j \leq k,$ we see that

\begin{equation}\label{2 leq j leq k}
    \frac{(1-x_1)x_1x_j}{(x_j-x_1)(1-x_1)(1-x_j)} + \frac{(1-x_j)x_1x_j}{(x_1-x_j)(1-x_1)(1-x_j)} = \frac{x_1x_j}{(1-x_1)(1-x_j)}.
\end{equation}

On the other hand, for $j > k,$
\begin{equation}\label{j>k}
    \frac{(1-x_1)x_1}{(x_j-x_1)(1-x_1)} + \frac{(1-x_j)x_j}{(x_1-x_j)(1-x_j)} = -1.
\end{equation}

Thus for each $w \in \mathfrak S_n$ and $j \in [2,n],$ we pair the summand $$w \cdot \left((1-x_1)\frac{x_1}{1-x_1}\frac{x_2}{1-x_2}\cdots\frac{x_k}{1-x_k}\cdot \frac{1}{x_j-x_1}\right)
$$ with the summand obtained by taking $w$ and transposing $w(1)$ and $w(j).$ Hence we duplicate the sum and simplify with equations \eqref{2 leq j leq k} and \eqref{j>k}:
\begin{align*}
    & 2\sum_{w \in \mathfrak S_n} w \cdot \left((1-x_1)\frac{x_1}{1-x_1}\frac{x_2}{1-x_2}\cdots\frac{x_k}{1-x_k}\sum_{j=2}^n \frac{1}{x_j-x_1}\right)U \\
    & = \sum_{w \in \mathfrak S_n} w \cdot \left((k-1) \frac{x_1}{1-x_1}\frac{x_2}{1-x_2}\cdots \frac{x_k}{1-x_k} - (n-k)\frac{x_2}{1-x_2}\cdots \frac{x_k}{1-x_k}\right) U\\
    & = (k-1) P_kU - (n-k)P_{k-1}U.
\end{align*}
Thus, the expression in equation \eqref{antisymmetrizing} simplifies to:
\begin{multline}
\sum_{w \in \mathfrak S_n} (-1)^w w \cdot \left(\frac{\partial}{\partial x_1}(1-x_1)\frac{x_1}{1-x_1}\frac{x_2}{1-x_2}\cdots\frac{x_k}{1-x_k}U\right) = \\  
 bP_kU + (1-a)P_{k-1}U - \frac{c}{2}(n-k)P_{k-1}U + \frac{c}{2}(k-1)P_kU.
\end{multline}
Since the residue of a partial derivative of an analytic function is always zero, taking the residues of the terms allows setting the equation to 0: 

\begin{equation*}
0\,=\, b\Res_xP_kU + (1-a)\Res_xP_{k-1}U - \frac{c}{2}(n-k)\Res_xP_{k-1}U + \frac{c}{2}(k-1)\Res_xP_kU.
\end{equation*}

By definition of $\Psi(\cdot),$ we have that $\Res_x P_kU = k!(n-k)!\Psi_n(k,a,b,c),$ which gives:

\begin{equation*}
 (b+(k-1)\frac{c}{2})k!(n-k)!\Psi_{n}(k,a,b,c) = (a-1+(n-k)\frac{c}{2})(k-1)!(n-k+1)!\Psi_{n}(k-1,a,b,c).
\end{equation*}
Simplifying gives relation  \eqref{psi rel 4} for odd $c.$ When $c$ is even, $U$ is symmetric, so symmetrizing over $\mathfrak S_n$ gives:
\begin{align*}
    & \sum_{w \in \mathfrak S_n} w \cdot (\frac{\partial}{\partial x_1}(1-x_1)\frac{x_1}{1-x_1}\frac{x_2}{1-x_2}\cdots\frac{x_k}{1-x_k}U)\\
    & = bP_kU + (1-a)P_{k-1}U + c \sum_{w \in \mathfrak S_n} w \cdot \left((1-x_1)\frac{x_1}{1-x_1}\frac{x_2}{1-x_2}\cdots\frac{x_k}{1-x_k}\sum_{j=2}^n \frac{1}{x_j-x_1}\right)U,
\end{align*}

which is essentially identical to when $c$ is odd, and the proof follows verbatim.
\end{proof}

\subsection{Closed Formula for \texorpdfstring{$\Psi_n(k,a,b,c)$}{Psi_n(k,a,b,c)}}

Our proof for the closed formula of $\Psi_n(k,a,b,c)$ follows the recurrence approach used by Baldoni-Vergne \cite[p. 10]{BV} (see also \cite[Prop. 3.11]{tesler}) using the recurrences proven in the previous section.

\begin{lemma} \label{psi uniquely defined}
The relations \eqref{psi rel 1}-\eqref{psi rel 4} uniquely determine the function $\Psi_n(k,a,b,c)$.
\end{lemma}

\begin{proof}

\noindent \textit{Case 1:} Consider if $c = 0, n \geq 1,$ and $a\geq 1.$ To compute $\Psi_n(k,a,b,0),$ we repeatedly apply equation \eqref{psi rel 4} to increment $k$ until $k = n,$ at which point we apply equation \eqref{psi rel 1}. Thus $\Psi_n(k,a,b,c)$ reduces to calculating $\Psi_n(0,a-1,b+1,c)$:

\[
\xymatrix{
\Psi_n(k,a,b,0)  \ar[r]^(.45){\eqref{psi rel 4}} & \Psi_n(k+1,a,b,0) 
\ar@{-->}[r]^{\eqref{psi rel 4}^*}
& \Psi_n(n,a,b,0)
\ar[r]^(.4){\eqref{psi rel 1}} & \Psi_n(0,a-1,b+1,0).
}
\]

By iterating this recursion, we see this is equivalent to calculating $\Psi_n(0,1,a+b-1,0).$ By equation \eqref{psi rel 3}, this is equal to 1.

\noindent \textit{Case 2:} Consider if $c \geq 1, n = 1,$ and $a \geq 1.$ Since $\prod_{1 \leq i < j \leq n} (x_j-x_i)^{-c}$ is the empty product, this is equivalent to when $c = 0,$ which implies that $$\Psi_{1}(k,c,a+b+c(n-2),c) = \Psi_{1}(k,c,a+b+c(n-2),0).$$

This reduces to Case 1.

\noindent \textit{Case 3:} Consider if $c \geq 1, n \geq 2,$ and $a \geq 1.$ Similar to in Case 2, to compute $\Psi_n(k,a,b,c)$ we repeatedly apply equation \eqref{psi rel 4} to increment $k$ until $k = n,$ at which point we apply equation \eqref{psi rel 1}. Thus $\Psi_n(k,a,b,c)$ reduces to calculating $\Psi_n(0,a-1,b+1,c).$ 

\[
\xymatrix{
\Psi_n(k,a,b,c)  \ar[r]^(.45){\eqref{psi rel 4}} & \Psi_n(k+1,a,b,c) 
\ar@{-->}[r]^(.55){\eqref{psi rel 4}^*}
& \Psi_n(n,a,b,c)
\ar[r]^(.4){\eqref{psi rel 1}} & \Psi_n(0,a-1,b+1,c).
}
\]

We iterate this recursion until $a = 1,$ at which point we reduce the calculation to finding $\Psi_n(0,1,a+b-1,c).$ Now we again increment $k$ by \eqref{psi rel 4} until $k = n-1$. Applying equation \eqref{psi rel 2} reduces the calculation to finding $\Psi_{n-1}(0, c,a+b, c).$
\[
\xymatrix{
\Psi_n(k,1,a+b-1,c) \ar[r]^(.46){\eqref{psi rel 4}} & \Psi_n(k+1,1,a+b-1,c) 
\ar@{-->}[r]^{\eqref{psi rel 4}^*}
& \Psi_n(n-1,1,a+b-1,c)
\ar[r]^(.53){\eqref{psi rel 2}} & \Psi_{n-1}(0,c,a+b,c).
}
\]
We now repeatedly apply the above two cycles until we reduce $n$ to 1, in which case we reduce the computation to $\Psi_1(0,c,a+b+c(n-2),c).$ Since $n = 1,$ this becomes Case 2.

Since all cases eventually reduce to case 1, the result follows.
\end{proof}

Using the fact that the relations \eqref{psi rel 1}-\eqref{psi rel 4} uniquely define $\Psi_n(k,a,b,c),$ we now prove our explicit product formula.

\begin{proof}[Proof of Theorem~\ref{thm:psi product}]
By Lemma \ref{psi uniquely defined}, it is sufficient to show the product formula for  $\Psi_n(k,a,b,c)$ in equation \eqref{eq: product formula Psi} satisfies the relations \eqref{psi rel 1}-\eqref{psi rel 4}.
To show equation \eqref{psi rel 1}, recall that:
\begin{equation*}
    M_n(a-1,b+1,c) = \prod_{j=0}^{n-1} \frac{\Gamma(a+b-1+(n-1+j)\frac{c}{2})\Gamma(\frac{c}{2}+1)}{\Gamma(a-1+j\frac{c}{2})\Gamma(b+1+j\frac{c}{2})\Gamma((j+1)\frac{c}{2}+1)}.
\end{equation*}

Recall that $\Gamma(x+1) = x\Gamma(x).$ Hence, $\Gamma(a-1+j\frac{c}{2}) = \Gamma(a+j\frac{c}{2})/(a-1+j\frac{c}{2}),$ and $\Gamma(b+1+j\frac{c}or{2}) = (b+j\frac{c}{2})\Gamma(b+j\frac{c}{2}).$ Substituting gives \begin{equation*}
    M_n(a-1,b+1,c) = M_n(a,b,c) \prod_{j=0}^{n-1} \frac{a-1+j\frac{c}{2}}{b+j\frac{c}{2}} = M_n(a,b,c)\prod_{j=1}^{n} \frac{a-1+(n-j)\frac{c}{2}}{b+(j-1)\frac{c}{2}}.
\end{equation*}

To show equation \eqref{psi rel 2} from \eqref{psi rel 1}, it is sufficient to show the product formula satisfies Lemma \ref{psi contraction}.

First we show $M_n(1,b,c) = M_{n-1}(c+1,b,c).$ To do so, consider the ratio \begin{equation*}
    \frac{M_n(1,b,c)}{M_{n-1}(c+1,b,c)} = \frac{1}{n} \cdot \frac{\Gamma(b+(n-1)\frac{c}{2})\Gamma(\frac{c}{2})}{\Gamma(b+(n-1)\frac{c}{2})\Gamma(\frac{c}{2}n)} \cdot \frac{\Gamma(\frac{c}{2}n+1)}{\Gamma(\frac{c}{2}+1)\Gamma(1)}
\end{equation*}

Since $\Gamma(x+1) = x\Gamma(x),$ the above ratio simplifies to 1, and the result follows.

Using the above equality, it is sufficient to show that $\binom{n}{k}\prod_{j=1}^k (n-j) = \binom{n-1}{k}\prod_{j=1}^k (n+1-j).$ Both sides of the equation simplify to $\prod_{j=1}^k (n+1-j)(n-j)/j,$ thus proving relation the product formula satisfies Lemma \ref{psi contraction} and hence equation \eqref{psi rel 2}.

To show equation \eqref{psi rel 3}, recall that since $k = 0,$ we have that $\binom{n}{k}\prod_{j=1}^k\frac{a-1+(n-j)\frac{c}{2}}{b+(j-1)\frac{c}{2}} =1$. Then
\begin{equation*}
    \Psi_n(0,1,b,0) = M_n(1,b,0) = 1.
\end{equation*}

To show equation \eqref{psi rel 4}, note that 

\begin{equation*}
    \frac{\Psi_n(k,a,b,c)}{\Psi_n(k-1,a,b,c)} = \frac{\binom{n}{k}}{\binom{n}{k-1}} \cdot \frac{a-1 + (n-k)\frac{c}{2}}{b+(k-1)\frac{c}{2}} = \frac{n-k+1}{k} \cdot \frac{a-1 + (n-k)\frac{c}{2}}{b+(k-1)\frac{c}{2}}. 
\end{equation*}

Rearranging gives the desired recurrence relation, and the result follows.
\end{proof}

We also compute the following special cases of $\Psi_n(k,a,b,c),$ which generalize the special cases of $M_n(a,b,c)$ computed in Section~\ref{sec: Morris}.

\begin{corollary} \label{psi special cases}
The constant term $\Psi_n(k,a,b,c)$ satisfies the following:
\begin{align} 
\label{eq: case ka11} \Psi_n(k,a,1,1) & = \frac{1}{n+2(a-1)}\binom{n}{k}\binom{n+2(a-1)}{k+1}M_n(a,1,1)\\
\label{eq: case k1b1} \Psi_n(k,1,b,1) & = \binom{n-1}{k}\binom{n}{k}\binom{k+2b-1}{k}^{-1}M_n(1,b,1)\\
\Psi_n(k,1,1,c) & = N(n,k+1)\prod_{j=1}^{k}\frac{c(j+1)}{c(j-1)+2}M_n(1,1,c).
\end{align}
\end{corollary}

\begin{proof}
We manipulate the product formula given in Theorem \eqref{thm:psi product} to obtain the desired relations. Note that
\begin{align*}
    \prod_{j=1}^{k}\frac{2(a-1)+(n-j)}{j+1} &= \frac{1}{n+2(a-1)}\binom{n+2(a-1)}{k+1} \\
    \prod_{j=1}^{k}\frac{n-j}{j+2b-1} &= \frac{\binom{n-1}{k}}{\binom{k+2b-1}{k}} \\
    \prod_{j=1}^{k}\frac{(n-j)c}{2+(j-1)c} &= \frac{1}{n}\binom{n}{k+1}\prod_{j=1}^{k}\frac{c(j-1)+2c}{c(j-1)+2}.
\end{align*}
The results follows from substitution.
\end{proof}

\begin{remark}
Note that unlike $M_n(a,b,c)$ (see Proposition~\ref{M symmetry}), in most cases $\Psi_n(k,a,b,c)\neq \Psi_n(k,b,a,c).$ For instance, by plugging into equations \eqref{eq: case ka11} and \eqref{eq: case k1b1}, we obtain $\Psi_2(1,2,1,1) = 6,$ whereas $\Psi_2(1,1,2,1) = 1$. Instead, we have the following symmetry.
\end{remark}

\begin{proposition}[Symmetry of $\Psi_n$]\label{prop: psi symmetry}
For positive integers $a, b, n$ and nonnegative integers $c$ and $k \leq n,$
$$\Psi_n(k,a,b,c) = \Psi_n(n-k,b+1,a-1,c).$$
\end{proposition}
\begin{proof}
By Theorem~\ref{psi kpf}~(ii), $\Psi_n(k,a,b,c)$ is the volume of the interior disjoint polytopes $\{ \mathcal{F}_{k_{n+2}^{a,b,c}(S)} \mid S \in \binom{[n]}{k}\}$. Note that the reverse graph of $k_{n+2}^{a,b,c}(S)$ is the graph $k_{n+2}^{b+1,a-1,c}([n]\setminus S)$. So by Lemma~\ref{lem: reverse flow polytopes} we have that  $\mathcal{F}_{k_{n+2}^{a,b,c}(S)} \equiv \mathcal{F}_{k_{n+2}^{b+1,a-1,c}([n]\setminus S)}$ and both polytopes have the same volume. Thus we have 
\begin{align*}
    \Psi_n(k,a,b,c) \,=\, \sum_{S\in \binom{[n]}{k}} \vol \mathcal{F}_{k_{n+2}^{a,b,c}(S)} = \sum_{S \in \binom{[n]}{n-k}} \vol \mathcal{F}_{k_{n+2}^{b+1,a-1,c}(S)}  = \Psi_{n}(n-k,b+1,a-1,c). 
\end{align*}
\end{proof}

\begin{remark}
Combining Proposition~\ref{prop: psi symmetry} with  Corollary~\ref{cor:M refinement} yields $M_n(a,b+1,c)=M_n(b+1,a,c)$, recovering the symmetry of the Morris identity (Corollary~\ref{M symmetry}).
\end{remark}

As a corollary, we also have the following special case when $a = b= c = 1.$ 

\begin{corollary}\label{Narayana refinement}
For positive integer $n$ and nonnegative integer $k \le n$, 
\[
\Psi_n(k,1,1,1) =  N(n,k+1)C_{n-1}\cdots C_1.
\]
\end{corollary}

\begin{proof}
This follows from \eqref{eq: case ka11} when $a=1.$\end{proof}

\begin{corollary}
Theorem~\ref{psi kpf} and Theorem~\ref{thm:psi product} imply Theorem 1.5.
\end{corollary}

\begin{proof}
Theorem~\ref{thm:psi product} implies Corollary~\ref{Narayana refinement}. Applying Corollary~\ref{a11 refinement} and the Kostant partition function interpretation in Theorem~\ref{psi kpf} thus gives Theorem 1.5.
\end{proof}

\begin{remark} 
Corollary~\ref{Narayana refinement} and Lemma \ref{psi polytope} in the case that $a = b= c = 1$ give a coarser version of the refinement provided by M\'esz\'aros \cite[Thm. 13]{Mproduct}.
\end{remark}

\section{The Baldoni-Vergne refinement of \texorpdfstring{$M_n(a,b,c)$}{M_n(a,b,c)}} \label{sec: phi}

To prove the Morris identity, Baldoni-Vergne defined the generating function  
\begin{equation*} \label{phi'}
\phi'_n(k,a,b,c) := k!(n-k)!e_k\prod_{i=1}^n (1-x_i)^{-b}x_i^{-a+1} \prod_{1 \leq i < j \leq n} (x_j - x_i)^{-c},
\end{equation*}
where $e_k = [t^k]\prod_{i=1}^n(1+tx_i)$ is the $k$th elementary symmetric polynomial. They proved several recurrence relations that computed the constant term $\Phi'_n(k,a,b,c) := \CT_x \phi'_n(k,a,b,c),$ which implies the Morris identity when $k = 0.$

\begin{theorem}[Baldoni-Vergne \cite{BV}] \label{thm:BV}
For positive integer $n$ and nonnegative integer $k,a,b,c,$ with $a + b \ge 2,$ the constant term $\Phi'_n(k,a,b,c)$ is given by the formula
\begin{equation} \label{eq: prod Phi}
    \Phi'_n(k,a,b,c) \,=\, n!\cdot  M_n(a,b,c) \prod_{j=1}^k \frac{a-1 + (n-j)\frac{c}{2}}{a+b-2 + (2n-j-1)\frac{c}{2}}.
\end{equation}
\end{theorem}

Interestingly, the Baldoni-Vergne constant term does not generalize the refinement of $M_n(1,1,1)$ given by Theorem~\ref{conj 2 cry} which helped motivate our new refinement of $M_n(a,b,c)$ in Section~\ref{sec: psi}. To more naturally interpret this constant term with Kostant partition functions, we scale $\Phi'_n(k,a,b,c)$.

\begin{definition}
We define the following modification of the Baldoni-Vergne constant term:
\begin{equation} \label{eq:definition phi}
    \Phi_n(k,a,b,c) \,:=\, \CT_x [t^k] \prod_{i=1}^n (1-x_i)^{-b}x_i^{-a+1}(1+tx_i)\prod_{1 \leq i < j \leq n} (x_j - x_i)^{-c}.
\end{equation}
\end{definition}
Equivalently, we have $\Phi_n(k,a,b,c) := \Phi'_n(k,a,b,c)/(k!(n-k)!).$ We now show the main result of this . section.

\begin{theorem} \label{phi kpf}
$\Phi_n(k,a,b,c)$ is the sum of Kostant partition functions $K_{k_{n+2}^{a,b,c}}(0, a_1, \dots, a_{n}, -\sum_{j=1}^n a_j)$ such that  $a-2+c(i-1) \leq a_i \leq a -1+ c(i-1),$ with $a_i = a-1 + c(i-1)$ for exactly $n-k$ values of $i\in [n]$. 
\end{theorem}

\begin{proof}
 Recall that specializing the generating function of equation \eqref{eq: kpf gen function kabc} gives the Morris constant term
\begin{equation*}
    M_n(a,b,c) = \CT_x \prod_{i=1}^n (1-x_i)^{-b}x_i^{-a+1}\prod_{1 \leq i < j \leq n} (x_j - x_i)^{-c}.
\end{equation*}
We see that $\Phi_n(k,a,b,c)$ has an additional term $[t^k]\prod_{i=1}^n(1+tx_i),$ which replaces $\CT_{x_i}$ with $[x_i^{-1}]$ for exactly $k$ values of $i,$ so this is equivalent to decreasing the net flow at these vertices by 1.

Recall from equation \eqref{eq: Mn is kpf and volume} that $M_n(a,b,c) = K_{k^{a,b,c}_{n+2}}(0,a_1, a_2, \dots, a_{n}, -\sum_{j=1}^n a_j)$, with $a_i = a-1+c(i-1)$ for all $i \in [n].$ Decreasing the net flow by one at exactly $k$ vertices gives the desired Kostant partition function interpretation. 

There are $k!(n-k)!$ ways of distinguishing the vertices based on their net flow, we also obtain the combinatorial interpretation for $\Phi'_n(k,a,b,c),$ and the result follows.
\end{proof}

\begin{remark}
We note that $\Phi_n(k,a,b,c)$ does not seem to have a refinement similar to Corollary~\ref{cor:M refinement}. Summing $\Phi(\cdot)$ over $k$ removes all restrictions on $t$ terms from the expression, giving:
\begin{equation*}
    \sum_{k=0}^n\Phi_n(k,a,b,c) \,:=\, \CT_x \prod_{i=1}^n (1-x_i)^{-b}x_i^{-a+1}(1+x_i)\prod_{1 \leq i < j \leq n} (x_j - x_i)^{-c},
\end{equation*}
for which a simplified expression is not immediate.
\end{remark}

We now give the recurrence relations used by Baldoni-Vergne \cite{BV} to prove Theorem~\ref{thm:BV}. These relations served as the inspiration for the relations for $\Psi_n(k,a,b,c)$ in Section~\ref{sec: psi}. 

\begin{proposition}[{Baldoni-Vergne \cite[Thm. 27]{BV}}] The constant term $\Phi'_n(k,a,b,c)$ is uniquely determined by the following relations:
\begin{enumerate}[label=(\arabic*)]
    \item $\Phi'_n(n,a,b,c) = \Phi'_n(0,a-1,b,c)$
    \item $\Phi'_n(n-1,1,b,c) = \Phi'_{n-1}(0,c,b,c)$
    \item $\Phi'_n(0,1,b,0) = r!$
    \item $\Phi'_1(k,0,b,c) = 0$
    \item $(a+b-2+\frac{c}{2}(2n-k-1))\Phi'_n(k,a,b,c) = (a-1+\frac{c}{2}(n-k))\Phi'_n(k-1,a,b,c)$.
\end{enumerate}
\end{proposition}

\begin{remark}
One can give combinatorial proofs for all but the last relation in a nearly identical manner to our combinatorial proofs for $\Psi_n(k,a,b,c)$ in Lemma~\ref{psi recurrences}.
\end{remark}

\section{Final remarks} \label{sec: remarks}

In this paper we investigated a symmetry and a refinement of the Morris identity with several combinatorial interpretations, including a certain sum of Kostant partition functions and the volume of a collection of polytopes. We demonstrated how these collections of polytopes subdivide the graph $k_{n+2}^{a,b,c},$ and proved a product formula for our refinement. We now give some possible avenues for future exploration.

\subsection{The recurring appearance of Aomoto's integral}

The original Morris constant term identity \cite{Morr} strongly resembles the Selberg integral, and the two identities are known to be equivalent (see \cite{Morr} and \cite{Forrester}). Interestingly, the product formula for the Baldoni-Vergne refinement of the Morris identity (equation \eqref{eq: prod Phi})  greatly resembles Aomoto's integral \cite{aomoto}. However, the relationship between these two seemingly related identities is as of yet unclear. Intriguingly, Zeilberger also cites Aomoto's integral in his proof of Conjecture 2 of Chan-Robbins-Yuen, and while we did not see an immediate application of Aomoto's integral in our proof of the product formula of $\Psi_n(k,a,b,c),$ this seems to suggest these refinements of the Morris identity are in some way related to Aomoto's generalization of the Selberg integral. For a recent bijective proof of the Selberg integral see  \cite{haupt}.

\subsection{Combinatorial proof of the Morris identity}

This paper provides multiple combinatorial proofs of recurrence relations for $\Psi_n(k,a,b,c)$ that could contribute to a combinatorial proof of the Morris constant term identity, and therefore, the volume formula for the Chan-Robbins-Yuen polytope. With the approach of this paper, the only remaining step is to give a combinatorial proof for equation \eqref{psi rel 4}. A combinatorialization of our algebraic proof of \eqref{psi rel 4}, or a new combinatorial proof altogether, would certainly be interesting. We also note that equation \eqref{psi rel 4} can be rewritten as
\[\left(kb+\binom{k}{2}c\right)\cdot \Psi_{n}(k,a,b,c) = \left((n-k+1)(a-1)+\binom{n-k+1}{2}c\right)\cdot \Psi_{n}(k-1,a,b,c).\]
where the extra factors on the left-hand and right-hand sides appear to be selecting certain edges of the graph $k_{n+2}^{a,b,c}$. Applying the identity in Proposition~\ref{prop: psi symmetry} to the right-hand side, the expression becomes
$$\left(kb+\binom{k}{2}c\right)\cdot \Psi_{n}(k,a,b,c) = \left((n-k+1)(a-1)+\binom{n-k+1}{2}c\right)\cdot \Psi_{n}(n-k+1,b+1,a-1,c).$$
where both sides have very similar structures. Given that a combinatorial proof of the Morris identity has been elusive and would serve immediately as a combinatorial proof for the volume formula of the Chan-Robbins-Yuen polytope. See \cite{benedetti2019combinatorial,jang2019volumes,yip2019fuss} for combinatorial proofs of volumes of flow polytopes $\mathcal{F}_G$ for other graphs $G$.

\subsection{Volume of polytopes with different net flow vectors}

In Section~\ref{sec:vol proof}, we presented a new recursive proof of Theorem~\ref{thm:vol-kpf}. Generalizing Theorem~\ref{thm:vol-kpf} is the following theorem of Baldoni-Vergne-Lidskii.
\begin{theorem}[Baldoni-Vergne-Lidskii \cite{BVkpf}]\label{thm: lidskii}
Let $G$ be a connected digraph on vertex set $\{0,1, \ldots, n,n+1\}$ with $m$ edges directed $i \to j$ if $i < j$ and such that for $i \in \{0, 1, \ldots n\},$ there is at least one outgoing edge at vertex $i.$ Then for a fixed net flow vector $\mathbf{a} = (\sum_{j=1}^n a_i,-a_1, -a_2,\ldots,-a_{n+1}),\;a_i \in \mathbb Z_{\geq 0},$ it holds that
$$\vol \mathcal F_G(\mathbf{a}) = \sum_\mathbf{j} \binom{m-n-1}{j_0, \ldots, j_{n+1}} a_1^{j_1}\cdots a_{n+1}^{j_{n+1}} \cdot K_G(0,d_1-j_1,\ldots,d_{n+1}-j_{n+1}),$$
where $d_i = \indeg_G(i) - 1.$ and the sum is over compositions ${\bf j}$ of $m-n-1$ with $n+1$ parts.
\end{theorem}

In our proof in Section~\ref{sec:vol proof} of Theorem~\ref{thm:vol-kpf}, the map $\varphi$ on Kostant partition functions that we introduced is not specific to flow polytopes with net flow vector $(1,0,\ldots,0,-1)$. This means that the inductive step will not change significantly for a different net flow vector, and as such, it is worth investigating whether there is a simple recursive proof for Theorem~\ref{thm: lidskii} considering new base cases with net flow vector $\mathbf{a}$. Such a proof would provide a better understanding of how volumes of flow polytope and Kostant partition functions are refined by the subdivision lemma. See \cite{kapoor2019counting} for another recent proof of this more general volume formula.

\subsection{Dual graph of triangulations of \texorpdfstring{$\mathcal{F}_G$}{F_G}}
 In Section~\ref{sec: symmetry} we used the Danilov-Karzanov-Koshevoy (DKK) triangulation of flow polytopes in terms of cliques of routes. Given a triangulation of a polytope, it is of interest to study its dual graph. This is the graph whose vertices are the top-dimensional simplices connected by an edge if the pair of simplices have a common facet (see \cite[Ch. 1]{triangulations}). In our context the number of vertices of such dual graph gives the volume of the polytope. In \cite{vonBell-et-al}, the authors show that for certain planar graphs $G$ and for two framings called {\em length} and {\em planar}, the dual graphs of the DKK triangulations of $\mathcal{F}_G$ are isomorphic to a generalization of the {\em Tamari lattice (associahedron)} and certain principal order ideals in {\em Young's lattice}, respectively. It would be of interest to study the  dual graph with $M_n(a,b,c)$ vertices of the DKK triangulation of the flow polytopes $\mathcal{F}_{k_{n+2}^{a,b,c}}$ for the length and planar framing.

\subsection{Triangulations of flow polytopes}
In Section~\ref{sec: symmetry} we used the  DKK triangulation of the  flow polytope $\mathcal{F}_G$ of a framed graph $(G,\preceq)$ to obtain a bijection $\Theta_{G,\preceq}$ between the integer flows of $\mathcal{F}_G(0,d_1,\ldots)$ and of $\mathcal{F}_{G^r}(0,d_1,\ldots)$. This is related (see Example~\ref{ex: prod simplices}) and was motivated by work of Postnikov \cite[Section 12]{Pos}, who showed that a triangulation $\tau$ of \emph{root polytopes} $Q_H$ for bipartite graphs $H$ with vertices $\{1,\ldots,p\} \cup \{\overline{1},\ldots,\overline{q}\}$ give a bijetion $\phi_{\tau}$ between lattice points of two {\em trimmed generalized permutahedra} $P^-_H$ and $P^+_{H^*}$, where $H^*$ is the obtained by flipping $H$. Galashin, Nenashev, Postnikov \cite{GNP} studied the bijections $\phi_{\tau}$ and showed that they uniquely specify the triangulation $\tau$. It would be of interest to do a similar study of the bijections $\Theta_{G,\preceq}$.

\section*{Acknowledgements}

We thank Sylvie Corteel for suggesting to the first author combinatorializing the proof of the product formula in \cite{BV} that inspired this project. The first author is grateful to AIM and the SQuaRE group
``Computing volumes and lattice points of flow polytopes” where some of the topics addressed here were discussed. We also thank William Dugan, Pavel Etingof, Rafael Gonz\'alez D'L\'eon, Tanya Khovanova, Gleb Nenashev, Boya Song, Igor Pak, Alex Postnikov, and Martha Yip for helpful comments and suggestions. This research was made possible by MIT PRIMES-USA 2020 program. Alejandro Morales was partially supported by the NSF Grant DMS-1855536.

\bibliography{references}
\bibliographystyle{plain}

\section*{Appendix}

In this section, we give some computational proofs for Section $\ref{sec: bg}.$ In multiple of the proofs below, we use the Legendre duplication formula:
\begin{equation} \label{eq:dup}
    \Gamma(x+\frac{1}{2})\Gamma(x) = 2^{1-2x}\sqrt{\pi}\cdot \Gamma(2x).
\end{equation}
We also use the following expression deducible from the Legendre duplication formula. For positive integers $x$ and $k,$
\begin{equation} \label{eq:dup2}
    \Gamma(x+k+1/2)\Gamma(x) = 2^{1-2x}\sqrt{\pi}\cdot \Gamma(2x) \prod_{j=0}^{k-1} (x+j+\frac{1}{2}).
\end{equation}

\begin{proof}[Proof of Corollary~\ref{cor: Mn(a,b,1)}]
First, consider the ratio $M_n(a,b,1)/M_{n-1}(a,b,1).$ By \eqref{eq:dup},

\begin{align*}
    \frac{M_{n}(a,b,1)}{M_{n-1}(a,b,1)} & = \frac{1}{n} \cdot \frac{\Gamma(a+b+n-\frac{5}{2})\Gamma(a+b+n-2)}{\Gamma(a+b-2+\frac{1}{2}n)} \cdot \frac{\Gamma(\frac{1}{2})}{\Gamma(a+\frac{n-1}{2})\Gamma(b+\frac{n-1}{2})\Gamma(\frac{n}{2})} \\
    & = \frac{1}{n} \cdot \frac{2^{6-2(a+b+n)}\Gamma(2(a+b+n)-5)\pi}{\Gamma(a+b-2+\frac{1}{2}n)\Gamma(a+\frac{n-1}{2})\Gamma(b+\frac{n-1}{2})\Gamma(\frac{n}{2})}.
\end{align*}

Substitution with \eqref{eq:dup2} then gives
\begin{align*}
    \frac{M_{n}(a,b,1)}{M_{n-1}(a,b,1)} & = \frac{(2(a+b+n)-6)!}{n!(2a+n-2)!\prod_{j=0}^{b-3} (2a+n+2j)\prod_{j=0}^{b-2} (n+2j+1)} \\ 
    & = \frac{(2(a+b+n)-6)!}{n!!(2a+n-3)!!(2b+n-3)!!(2a+2b+n-6)!!}.
\end{align*}

To cancel the double factorials, we instead consider the ratio $M_{n+1}(1,1,1)/M_{n-1}(1,1,1)$:

\begin{align}
    \frac{M_{n+1}(a,b,1)}{M_{n-1}(a,b,1)} &= \frac{(2(a+b+n)-4)!(2(a+b+n)-6)!}{(n+1)!(2a+n-2)!(2b+n-2)!(2a+2b+n-5)!} \notag \\
    &= \frac{\binom{2a+2b+2n-4}{2a+n-2}}{\binom{2a+2b+2n-4}{n}} C_{n-1}C_{n}\prod_{i=1}^{n-1}\frac{2(a+b-2) + n + i - 1}{n + i - 1}\prod_{i=1}^{n}\frac{2(a+b-2) + n + i}{n + i} \label{eq:ab1 recurrence}.
\end{align}

We establish the base cases $M_0(a,b,1) = 1,$ $M_1(a,b,1) = \binom{a+b-2}{a-1},$ so by telescoping, we see that:
\begin{align*}
    M_{2n}(a,b,1) & = \prod_{i=1}^n \frac{M_{2i}(a,b,1)}{M_{2i-2}(a,b,1)}\\
    M_{2n-1}(a,b,1) & = \binom{a+b-2}{a-1}\prod_{i=1}^{n-1} \frac{M_{2i+1}(a,b,1)}{M_{2i-1}(a,b,1)}.
\end{align*}
Plugging in with \eqref{eq:ab1 recurrence} gives the desired result.
\end{proof}

The remaining proofs in this section follow a similar scheme, using $M_0(a,b,c) = 1.$

\begin{corollary}
For $c$ even, we have that:
\begin{equation*}
    M_n(1,1,c) = \prod_{i=1}^n \binom{(2i-3)\frac{c}{2}}{(i-1)\frac{c}{2}}\binom{(2i-2)\frac{c}{2}}{(i-1)\frac{c}{2}}\binom{i\frac{c}{2}}{(i-1)\frac{c}{2}}^{-1}.
\end{equation*}
\end{corollary}

\begin{proof}
Again, consider $M_n(1,1,c)/M_{n-1}(1,1,c).$ We see that
\begin{align*}
    \frac{M_{n}(1,1,c)}{M_{n-1}(1,1,c)} &= \frac{1}{n}\cdot \frac{((2n-3)\frac{c}{2})!((2n-2)\frac{c}{2})!}{((n-2)\frac{c}{2})!} \cdot \frac{(\frac{c}{2}-1)!}{((n-1)\frac{c}{2})!^2(\frac{c}{2}n-1)!} \\
    &=\frac{((2n-3)\frac{c}{2})!((2n-2)\frac{c}{2})!}{((n-2)\frac{c}{2})!} \cdot \frac{(\frac{c}{2})!}{((n-1)\frac{c}{2})!^2(\frac{c}{2}n)!}.
\end{align*}

With some rearrangement, we get the equation
$$\frac{M_n(1,1,c)}{M_{n-1}(1,1,c)} = \binom{n\frac{c}{2}}{(n-1)\frac{c}{2}}^{-1}\binom{(2n-3)\frac{c}{2}}{(n-1)\frac{c}{2}}\binom{(2n-2)\frac{c}{2}}{(n-1)\frac{c}{2}}.$$

Since $M_n(1,1,c) = \prod_{i=1}^n M_i(1,1,c)/M_{i-1}(1,1,c),$ the result follows.
\end{proof}

\begin{corollary}
For $c$ odd, we have that:
\begin{equation*}
    M_n(1,1,c) = \prod_{i=1}^n\frac{(1+(2i-3)c)!((i-1)c)!c!!}{((i-2)c)!!((i-1)c)!!^2(ic)!!((2i-3)\frac{c}{2}+\frac{1}{2})!}.
\end{equation*}
\end{corollary}
\begin{proof}
Let $k = \lfloor \frac{c}{2} \rfloor$ or $\frac{c}{2} = k + \frac{1}{2}$. We then simplify the ratio $M_n(1,1,c)/M_{n-1}(1,1,c)$ using \eqref{eq:dup2}.

\begin{align*}
    \frac{M_n(1,1,c)}{M_{n-1}(1,1,c)}& = \frac{1}{n}\cdot \frac{\Gamma(1 + (2n-3)\frac{c}{2})\Gamma(1 + (2n-2)\frac{c}{2})}{\Gamma(1 + (n-2)\frac{c}{2})} \cdot \frac{\Gamma(\frac{c}{2})}{\Gamma(1+(n-1)\frac{c}{2})^2\Gamma(\frac{c}{2}n)}\\
& = \frac{1}{n} \cdot\frac{2^{c(3-2n)-1}\Gamma(2+(2n-3)c) \prod_{j=0}^{k-1} (1+(2n-3)\frac{c}{2}+j+\frac{1}{2})}{2^{c(2-n)-1}\Gamma(2+(n-2)c) \prod_{j=0}^{k-1} (1+(n-2)\frac{c}{2}+j+\frac{1}{2})} \cdot \frac{\Gamma(\frac{c}{2})}{\Gamma(1+(n-1)\frac{c}{2})\Gamma(\frac{c}{2}n)} \\
& = 2^{1-c}\cdot \frac{(1+(2n-3)c)!c!!}{(1+(n-2)c)!(1+(n-1)c)!}\prod_{j=0}^{k-1}\frac{1+(2n-3)\frac{c}{2}+j+\frac{1}{2}}{(1+(n-2)\frac{c}{2}+j+\frac{1}{2})(1+(n-1)\frac{c}{2}+j+\frac{1}{2}))} \\
& = 2^{1-c+k} \cdot \frac{((2n-3)c)!!c!!((2n-2)c)!!}{((n-2)c)!!((n-1)c)!!^2(nc)!!} = \frac{(1+(2n-3)c)!((n-1)c)!c!!}{((n-2)c)!!(((n-1)c)!!)^2(nc)!!((2n-3)\frac{c}{2}+\frac{1}{2})!}.
\end{align*}
We have $M_n(1,1,c) = \prod_{i=1}^n M_i(1,1,c)/M_{i-1}(1,1,c),$ and the result follows.
\end{proof}

\begin{proof} [Proof of Corollary~\ref{cor: Mn(a,b,2k)}]
We again compute the ratio $M_n(a,b,2k)/M_{n-1}(a,b,2k)$:
\begin{align*}
    \frac{M_{n}(a,b,2k)}{M_{n-1}(a,b,2k)} &= \frac{1}{n}\cdot \frac{\Gamma(a + b - 1 + (2n-3)k)\Gamma(a + b -1 + (2n-2)k)}{\Gamma(1 + (n-2)k)} \cdot \frac{\Gamma(k)}{\Gamma(a+(n-1)k)\Gamma(b+(n-1)k)\Gamma(kn)} \\
    &= \frac{(a + b - 2 + (2n-3)k)!(a + b -2 + (2n-2)k)!}{((n-2)k)!} \cdot \frac{k!}{((a-1)+(n-1)k)!((b-1)+(n-1)k)!(kn)!} \\
    &= \frac{(a+b-2+(2n-3)k)!k!}{((n-2)k)!(nk)!} \cdot \binom{a + b-2+(2n-2)k}{a-1 + (n-1)k}.
\end{align*}
As with the above proofs, $M_n(a,b,2k) = \prod_{i=1}^n M_i(a,b,2k)/M_{i-1}(a,b,2k),$ so the result follows.
\end{proof}

\subsection{Asymptotic Analysis} label{sec:asymptotics} In this subsection, we examine the asymptotics of the Morris identity product formula. We use the standard asymptotics notations $f \sim g$ and $f = O(g).$

\begin{lemma} \label{lemma: cat asymptotics} We have the following asymptotic formula:
$$\log M_n(1,1,1)  = n^2 \log 2 - \frac{3}{2}n\log n + O(n).$$
\end{lemma}
\begin{proof}
It is well known that $C_n \sim 4^n/(n^{3/2} \sqrt \pi).$ Then
$$\log M_n(1,1,1) = \sum_{i=1}^{n-1} \log C_i \sim \sum_{i=1}^{n-1} i \log 4 - \frac{3}{2} \log i - \frac{1}{2} = \binom{n}{2} \log 4 - \frac{n-1}{2} - \frac{3}{2} \log ( (n-1)!)$$
By Stirling's formula and some manipulation, the result follows.
\end{proof}

\begin{proposition}\label{asymptotic Mna11} For positive integer $n$, 
$\log M_n(n,1,1) = (9 \log 2 - \frac{9}{2} \log 3)n^2 + O(n\log n).$
\end{proposition}

\begin{proof}
By \cite{morales2018asymptotics}, we have an alternative formulation of Proctor's formula
$$\prod_{1 \le i < j \le n} \frac{2(n-1)+i+j-1}{i+j-1} = \frac{\gimel(2(n-1)+2n)\gimel(2(n-1)+1)\Lambda(n)}{\Lambda(2(n-1)+n)\gimel(2n)}$$
where $\Lambda(n) = 1! \cdot 2! \cdots (n-1)!$ and $\gimel(n) = (n-2)!(n-4)!\cdots$. Then
\begin{align*}
    \log \prod_{1 \le i < j \le n} \frac{2(n-1)+i+j-1}{i+j-1} & = \log \gimel(4n-2) + \log \gimel(2n-1)+ \log \Lambda(n) - \log \Lambda(3n-2) - \log \gimel(2n) \\
     & = \log \gimel(4n) + \log \Lambda(n) - \log \Lambda(3n) +O(n \log n).
\end{align*}
In \cite{morales2017hook}, we have the asymptotics
$$\log \Lambda(n) = \frac{1}{2} n^2 \log n - \frac{3}{4}n^2 + O(n \log n)$$
$$\log \gimel(n) = \frac{1}{4} n^2 \log n - \frac{3}{8} n^2 + O(n\log n).$$
 Then substituting with some manipulation gives that the above sum is equal to
$$(8 \log 2 - \frac{9}{2} \log 3)n^2 + O(n\log n).$$ 

The result follows by applying Lemma~\ref{lemma: cat asymptotics}.
\end{proof}

\begin{proposition}
For positive integer $n$ and fixed positive integer $a$ and $b,$ we have that
$$\log M_{n}(n,n,1) = 2n^2\log n + (3+ 13 \log 2 + \frac{9}{2} \log 3 - \frac{25}{4} \log 5)n^2 + O(n\log n).$$
\end{proposition}

\begin{proof}
With some manipulation, we can show that
$$M_{n}(a,b,1) \le \binom{a+b-2}{a-1} \frac{\gimel(n)\gimel(2a+2b+n-4)}{\gimel(2a+n-2)\gimel(2b+n-2)} M_n(1,1,1)\prod_{1 \leq i < j \leq n-1}\frac{2(a+b-2) + j + i - 1}{j + i - 1},$$
where equality holds when $n$ is odd. The lower bound occurs by removing the $\binom{a+b-2}{a-1}$ term. Thus
\begin{align*}
    \log M_{n}(n,n,1) & = \log \frac{\gimel(n)\gimel(5n-4)}{\gimel(3n-2)\gimel(3n-2)} + \log M_{n}(2n-1,1,1) + O(n\log n) \\
    & = \log \frac{\gimel(n)\gimel(5n)}{\gimel(3n)^2} + \log M_{n}(2n-1,1,1) + O(n\log n).
\end{align*}
The first term on the right-hand side is 
$$\log \frac{\gimel(n)\gimel(5n)}{\gimel(3n)^2} = 2n^2 \log n + (\frac{25}{4} \log 5 - \frac{9}{2} \log 3 +3)n^2 + O(n\log n).$$
The second term is
\begin{align*}
    \log \frac{\gimel(6n)\gimel(4n)\Lambda(n)}{\Lambda(5n)\gimel(2n)} + \log M_n(1,1,1) + O(n\log n) & = (13 \log 2 + 9 \log 3 - \frac{25}{2} \log 5)n^2 + O(n\log n).
\end{align*}
The result follows from summing the two terms.
\end{proof}

\end{document}